\documentclass[a4paper,10pt, oneside]{report}
\usepackage{amsmath, amsthm, amscd, amsfonts}
\usepackage{graphicx,float}
\usepackage{amsmath}
\usepackage{amsthm}
\usepackage{amsfonts}
\usepackage{amssymb}
\usepackage{mathrsfs}
\usepackage[T1]{fontenc}
\usepackage[latin1]{inputenc}
\usepackage{epsfig}
\usepackage{setspace}
\usepackage{titlesec}

\usepackage{amssymb}

\usepackage{pb-diagram}
\newcommand{\PP}{{\mathbb{P}}}
\newcommand{\QQ}{{\mathbb{Q}}}
\newcommand{\RR}{{\mathbb{R}}}
\newcommand{\SSS}{{\mathbb{S}}}
\newcommand{\NN}{{\mathbb{N}}}
\newcommand{\MM}{{\mathbb{M}}}

\newcommand{\Card}{\ensuremath{\text{Card}}}

\DeclareMathOperator{\len}{lh}

\DeclareMathOperator{\Force}{Force}

\DeclareMathOperator{\dom}{dom}

\DeclareMathOperator{\Col}{Col}

\DeclareMathOperator{\supp}{supp}

\DeclareMathOperator{\Suc}{Suc}
\DeclareMathOperator{\Lev}{Lev}

\DeclareMathOperator{\range}{range}

\def\k{\kappa}

\def\l{\lambda}

\def\a{\alpha}

\def\d{\delta}
\def\b{\beta}

\newtheorem{theorem}{Theorem}[section]
\newtheorem{lemma}[theorem]{Lemma}

\newtheorem{corollary}[theorem]{Corollary}
\newtheorem{notation}[theorem]{Notation}
\newtheorem{definition}[theorem]{Definition}

\newtheorem{remark}[theorem]{Remark}
\newtheorem{claim}[theorem]{Claim}
\numberwithin{equation}{section}

\usepackage{pb-diagram}

\def\l{\lambda}

\def\rmark{\mbox{$\rm\bf\rule{0.06em}{1.45ex}\kern-0.05em R$}}
\def\pmark{\mbox{$\rm\bf\rule{0.06em}{1.45ex}\kern-0.05em P$}}
\def\nmark{\mbox{$\rm\bf\rule{0.06em}{1.45ex}\kern-0.05em N$}}
\newcommand{\lusim}[1]{\smash{\underset{\raisebox{1.2pt}[0cm][0cm]{$\sim$}}
{{#1}}}}

\begin{document}

\begin{center}

\end{center}
\begin{center}

\end{center}

\begin{center}

\end{center}
\begin{center}

\end{center}
\begin{center}

\end{center}
\begin{center}

\end{center}

\begin{center}

\end{center}
\begin{center}
\LARGE{ \textbf{AN INTRODUCTION TO FORCING}}
\end{center}
\begin{center}

\end{center}
\begin{center}
\LARGE{\textrm{Mohammad Golshani}}\footnote{School of Mathematics, Institute for Research in Fundamental Sciences (IPM), P.O. Box:
19395-5746, Tehran-Iran.

The author thanks Rahman Mohammadpour for his help in typing the notes.

None of the results or proofs presented here are due to the author; however no references are given.

E-mail address: golshani.m@gmail.com}
\end{center}

\newpage
\tableofcontents
\chapter{How to use forcing}
The aim of these lectures is to give a short introduction to forcing. We will avoid metamathematical issues as much as possible and similarly we will avoid performing the actual construction of forcing. We assume familiarity with basic predicate logic, the axioms of $ZFC$ set theory and constructible sets. We will also make use of tools like the coding of Borel sets and the Shoenfield absoluteness result.

\section{Inner models and generic sets}
We will use naive set theory as a setting. In this framework, we can prove results about consistency by looking at models of set theory.
\begin{definition}
An inner model of $ZF$  is a class $M$ such that:
\begin{enumerate}
\item $M$ is a class of $V$, that is the axioms of $ZF$ are still valid (in $V$) if one applies replacement to formulas including one unary predicate $U$ interpreted by $M$,
\item $M$ is transitive,
\item $M$ contains all ordinals,
\item $M$ is a model of $ZF$.
\end{enumerate}
\end{definition}
Similarly we can define when $M$ is an inner model of $ZFC$.
\begin{definition}
$(a)$ A forcing notion is a partially ordered set $\PP$ which has the largest element $1_{\PP}$; elements of $\mathbb{P}$ are called conditions.

$(b)$ Given $p,q\in \PP,$ $p$ is an extension of $q$ if $p\leq q.$

$(c)$ A subset $G$ of $\mathbb{P}$ is called $\mathbb{P}$-generic over $V$, if the following hold:
\begin{enumerate}
\item $p\leq q$ and $p\in G \Rightarrow q\in G,$
\item $p, q\in G \Rightarrow p, q$ are compatible (i.e. have a common extension),
\item If $D$ is a dense set belonging in $V$, then $D\cap G\neq \emptyset,$ where dense means $\forall p \exists q\leq p, q\in D.$
\end{enumerate}
\end{definition}
It is easily seen that if $G$ is $\mathbb{P}$-generic over $V$, and if $p,q\in G,$ then they have a common extension in $G$.
\begin{theorem}
If $M$ is a countable transitive model and $\mathbb{P}$ a partially ordered set of $M$, then given any condition in $\mathbb{P},$ there is a $\mathbb{P}$-generic set over $M$ including $p$ as an element.
\end{theorem}
\begin{proof}
Enumerate the dense sets of $\mathbb{P}$ in $M$ as a sequence $(D_n: n<\omega).$ Pick a decreasing sequence $(p_n: n<\omega)$ of elements of $\PP$ such that:
\begin{itemize}
\item $p_0=p,$
\item $p_{n+1}\leq p_n$,
\item $p_{n+1}\in D_n.$
\end{itemize}
Then $G=\{p\in \PP: \exists n, p_n\leq p \}$ is as required.
\end{proof}
The countability of the model is only used in the proof of the above theorem; so from now on we work in $V$, and force over it.

Fix a forcing notion $\PP$. We will use so called formulas with parameter $\PP,$ to mean a formula of an extended language including a constant symbol interpreted by $\PP.$
\\
{\bf  Construction of the model:} For any $\PP$-generic $G$ over $V$, there is a model $V[G]$ such that:
\begin{itemize}
\item $V$ is an inner model of $V[G]$,
\item There is an onto map $K_G$ from $V$ onto $V[G]$ defined in $V[G]$ with parameter $G$ (provided a unary predicate symbol is allowed with interpretation $V$).
\end{itemize}
An element $a$ such that $K_G(a)=u$ is called a name for $u$.
\\
{\bf Truth in the model:}
\begin{itemize}
\item For any formula $\phi(v_1, ..., v_n),$ there is a formula $\Force_{\phi}(v_0, ..., v_n)$ with parameter $\PP$ such that
\begin{center}
$V\models$``$\Force_{\phi}(p, a_1, ..., a_n)$''
\end{center}
iff for every generic set $G$ containing $p$,
\begin{center}
$V[G]\models$``$\phi(K_G(a_1)..., K_G(a_n))$''.
\end{center}
$\Force_{\phi}(p, a_1..., a_n)$ is often written
\begin{center}
$p\Vdash \phi(a_1, ..., a_n).$
\end{center}
Also we have
\begin{center}
$V[G]\models$``$\phi(K_G(a_1)..., K_G(a_n))$'',
\end{center}
iff
\begin{center}
$\exists p\in G, p\Vdash \phi(a_1, ..., a_n).$
\end{center}
\end{itemize}
Thus there is, in $V$, a forced approximation of the truth of $V[G].$
\\
{\bf  Names of elements of $V[G]$:} Recall that a name for $u$ is an element $a\in V$ such that $K_G(a)=u$.
\begin{itemize}
\item There is an object $\Gamma$ such that $K_G(\Gamma)=G$ (a canonical name for $G$).
\item There is a functional relation defined in $V, a \mapsto \check{a}$ such that $K_G(\check{a})=a.$
\end{itemize}
Most of the applications of forcing can be done without knowing more about generic models and the forcing relation.
\begin{theorem}
$V[G]$ is the smallest model containing all members of $V$ and $G$ as an element, and such that $V$ is an inner model.
\end{theorem}
\begin{notation} Let $V[G]$ be a generic extension of $V$.

$(a)$ For $a\in V[G],$ we use $\lusim{a}\in V$ as a name for $a$ (so that $K_G(\lusim{a})=a).$

$(b)$ If $a\in V,$ we use $a$ itself, instead of $\check{a},$ as a name for $a$.
\end{notation}

\section{Properties of the forcing relation and the generic extension}
In this section we give some consequences of the forcing relation and the model $V[G]$.
\begin{lemma}
If $p\Vdash \phi$ and $q\leq p,$ then $q\Vdash \phi.$
\end{lemma}
\begin{lemma}
$(a)$ $p\nVdash\phi$ iff $\exists q\leq p, q\Vdash \neg \phi.$

$(b)$  $p\Vdash \neg\phi$ iff $\forall q\leq p, q\nVdash \phi.$

$(c)$ $p\Vdash \forall x \phi(x)$ iff $p\Vdash \phi(a)$ for any $a$ in $V$.

$(d)$ $p\Vdash \exists x \phi(x)$ implies $\exists q\leq p, \exists t, q\Vdash \phi(t)$
\end{lemma}
\begin{proof}
$(a)$ Some model $V[G]$ with $p\in G$ satisfies $\neg \phi;$ hence assume $q\in G$ such that $q\Vdash \neg \phi.$ An extension $r$ of $p,q$ is smaller than $p$ and forces $\neg \phi.$

For the converse, pick a generic $G$ containing $q$ with $q\Vdash \neg \phi;$ then in the model $V[G], \neg \phi$ holds, hence $p$ can not force $\phi.$

$(b)$ follows from $(a)$,

$(c)$ If $p \nVdash \phi(a),$ some extension $q$ of $p$ forces $\neg \phi(a),$ by picking some generic $G$ with $q\in G,$ one comes to a contradiction.

For the converse, given $G$ with $p\in G,$ we get for any $a, V[G]\models \phi(K_G(a)),$ therefore $p\Vdash \forall x \phi(x)$.

$(d)$ Let $G$ be generic with $p\in G$. Then $V[G]\models$``$\exists x \phi(x)$'', thus for some $t, V[G]\models$``$\phi(K_G(t))$''. Pick $q\in G$ such that $q\Vdash \phi(t)$. Then any $r$ extending both of $p,q$  forces $\phi(t).$
\end{proof}
\begin{theorem}
If $V$ satisfies $AC$, then so does $V[G].$
\end{theorem}
\begin{proof}
We will well order a set $x$ of $V[G].$ Now every element of $x$ has a name:
\begin{center}
$\forall y\in x \exists b, y=K_G(b).$
\end{center}
This is a statement in $V[G].$ Given $y$, we can consider the first ordinal $\xi$ such that
\begin{center}
$\exists b\in V_\xi, y=K_G(b).$
\end{center}
By replacement we bound the search for the names. Now $K_G$ is an onto map from a well-ordered set onto a set that contains $x$; hence $x$ is well-orderable.
\end{proof}
\begin{definition}
$\PP$ satisfies the $\k$-c.c. if all antichains of $\PP$ have size $<\k,$ where an antichain $A$ is a subset of $\PP$ consisting of pairwise incompatible elements.
\end{definition}
\begin{theorem}
(Assume $V$ satisfies $AC$) If $\PP$ satisfies the $\k$-c.c. where $\k$ is regular, then forcing with $\PP$ preserves all cardinals $\geq \k.$
\end{theorem}
\begin{proof}
Assume not; so there is one, say $h:\l \leftrightarrow \l^+,$ for some regular $\l\geq \k.$ Some $p$ in $G$ forces
\begin{center}
``$\lusim{h}$ is a function from $\l$ onto $\l^+$''.
\end{center}
Given $\a<\l,$ pick a maximal antichain $A_\a$ consisting of conditions $q$ such that $q\leq p$ and $q\Vdash$``$\lusim{h}(\a)=\delta$'', for some $\delta<\l^+.$

Given any $q, \delta$ is unique. The set of possible $\delta$'s  is therefore of cardinality $<\k,$ as one has $\l$ many $\a$'s, this gives at most $\l$ possible $\d$'s altogether. Let $X$ be the set of these $\delta$'s.
\begin{claim}
In $V[G]$, the range of $h$ is included in $X$.
\end{claim}
\begin{proof}
Otherwise, pick $\b< \l, \rho\notin X$ such that $h(\beta)=\rho.$ Pick $q\leq p$ such that $q\Vdash$``$\lusim{h}(\beta)=\rho$''. $A_\beta$ is maximal, so $q$ is compatible with some $q'\in A_\beta.$ A common lower bound $r$ of $q, q'$ forces
\begin{center}
$r\Vdash$``$\lusim{h}(\beta)=\rho$'', $\rho\notin X,$

$r\Vdash$``$\lusim{h}(\beta)=\delta$'', $\delta\in X.$
\end{center}
Contradiction
\end{proof}
It follows that  the range of $h$ can not cover $\l^+$.
\end{proof}

\chapter{Random forcing}
\section{Adding one random real}
Let's start with the definition of the forcing notion.
The random (real) forcing $\RR$ is the set of compact sets of the real line of measure $> 0.$
\begin{lemma}
The forcing $\RR$ has the $c.c.c.$ (countable chain condition): any antichain is countable.
\end{lemma}
\begin{proof}
Define a semi-metric $d$ on the set of compact subsets of the real line by
\begin{center}
$d(K, K')=\mu(K\triangle K'),$
\end{center}
where $\mu$ is the lebesgue measure and $\triangle$ is the symmetric difference. In the associated topology, there is a countable dense set  namely the finite union of closed intervals. Indeed let $K$ be given, $K$ is covered by an open set $U$ with $\mu(U\setminus K) < \epsilon / 2,$ and a finite union of intervals $V=\bigcup_{i=1}^{n}(a_i, b_i),$ such that $V \subseteq U$ and $\mu(U\setminus V)< \epsilon/2$. Therefore $\mu(K\triangle \bar{V}) < \epsilon,$ and $\bar{V}$ is of required type.

Now if $2\epsilon <\mu(K_0)$ and $d(K, K_0)<\epsilon, d(C, K_o)<\epsilon,$ then $K, C$ are compatible. From this it follows that there is a countable basis of the topology consisting of sets $\mathcal{C}_n$ such that any two elements in  $\mathcal{C}_n$ are compatible. The $c.c.c.$ easily follows.
\end{proof}
 Let $G$ be generic for the above set of conditions. The intersection of all compact sets in $G$ is a real.
\begin{remark}
Actually the compact sets do not remain compact in $V[G]$. We replace them by their closure.
\end{remark}
The uniqueness is proved as follows: If not, let $g, g'$ be two elements of the intersection. Let $q\in\QQ, g < q < g'.$ Now we claim that
\begin{center}
$\{K: K \subseteq (-\infty, q)$ or $ K \subseteq (q, +\infty)  \}$
\end{center}
is dense, therefore a generic cannot contain both of $g$ and $g'$.
\begin{lemma}
The real $g$ does not belong to any $G_\delta$ set $X$ of zero measure coded in $V$.
\end{lemma}
\begin{proof}
 Let $X=\bigcap_{n<\omega}U_n,$ where each $U_n$ is an open set, $(U_n: n<\omega)$ is decreasing and $\mu(U_n)\rightarrow 0.$ We then note that
\begin{center}
$\{K: \exists n, K\cap U_n=\emptyset  \}$
\end{center}
is dense. This is because given $K_0$, we can pick $n$ such that $\mu(K_0\cap U_n) < \mu(K_0)/2.$ Then $K_0\setminus U_n$ is a compact set, if it is of measure $>0.$ From this the result follows immediately.
\end{proof}
We have a converse: Let $g$ be a real; let
\begin{center}
$\tilde{g}=\{K: K$ is a compact set coded in $V$ and $g\in \bar{K}\}.$
\end{center}
\begin{lemma}
$\tilde{g}$ is generic iff $g$ does not belong to any $G_\delta$ zero measure subset of $\mathbb{R}$ coded in $V$.
\end{lemma}
\begin{proof}
We have only one implication to establish. Properties $(1)$ and $(2)$ of genericity are clear. Let us see the third one. Let $D$ be a dense set in $V$. Pick a maximal antichain $A$ of elements of $D.$ $A$ is countable.
\begin{claim}
$\bigcup\{ K: K\in A\}$ is an $F_\sigma$ set whose completion is of zero measure.
\end{claim}
\begin{proof}
Otherwise some $K'$ is included in the complement with $\mu(K')>0.$ Replacing a smaller one $\tilde{K},$ we can assume $\tilde{K}\in D.$ This contradicts the maximality of $A$.
\end{proof}
Now the real $g$ does not belong to the complement of the set $\bigcup\{ K: K\in A\}$ of $V[G],$ hence for some $K, g\in \bar{K}.$
\end{proof}
It should be noted that if we go from $G$ to $g$ and then go to $\tilde{g},$ we get $G \subseteq \tilde{g}.$ Equality then follows from the following general lemma.
\begin{lemma}
If $G, G'$ are both $\PP-$geenric over $V$ and $G \subseteq G'$, then $G=G'.$
\end{lemma}
\begin{proof}
If $p\in G'\setminus G,$ then the set
\begin{center}
$D_p=\{q\in \PP: q\leq p$ or $q$ is incompatible with $p \}$
\end{center}
is dense, hence $G\cap D_p\neq \emptyset.$ Pick $q\in G\cap D_p.$ If $q\leq p, p\in G,$ contradiction. Otherwise $q$ is incompatible with $p$, then as $p,q$ are both in $G'$, we also get a contradiction.
\end{proof}
\begin{lemma}
Any real $x$ of $V[G]$ is the value on $g$ of a Borel measurable function of $V$.
\end{lemma}
\begin{proof}
We only treat the case of reals of the interval $[0,1]$; by adding a positive or negative integer it is possible to restrict ourself to this case. We first pick a condition $K_0$ such that
\begin{center}
$K_0\Vdash$``$\lusim{x}$ is a real of $[0,1]$'',
\end{center}
Now for any element $q$ of $\QQ\cap [0,1],$ pick a maximal antichain $A_q$ consisting of conditions $K\leq K_0$ such that $K\Vdash$``$\lusim{x} < q$''. $A_q$ is countable and we let $X_q=\bigcup\{ K: K\in A_q \}.$ $X_q$ is an $F_\sigma$ subset of $\RR$. We let $\Phi_q$ be the function whose value is $q$ on $X_q$ and is $1$ otherwise. Finally we define $\Phi$ to be $\inf_{q\in \QQ}\Phi_q.$
\begin{claim}
The value of $\Phi$ at $g$ is exactly $x$.
\end{claim}
\begin{proof}
First we show that $\Phi(g)\leq x.$ Otherwise, there is $q\in \QQ$ such that $x< q < \Phi(g).$ Now some condition $L$ of $G$ is such that $L\leq K_0$ and $L\Vdash$``$\lusim{x} < q$''. Now it is easily seen that the set
\begin{center}
$D=\{L': L'$ is incompatible with $L$, or $L'$ is below $L$ and some condition from $A_q  \}$
\end{center}
is dense. We pick some $L'\in G\cap D;$ $L'$ is a subset of $X_q$ and therefore  $\Phi(g)<q,$ contradiction.

Now we show that $x\leq \Phi(g).$ Otherwise for some $q, \Phi_q(g) < x.$ This implies $g\in X_q,$ hence $\Phi_q(g)=q.$ But then $g$ belongs to some $K\in A_q,$ contradiction as $K\Vdash$``$\lusim{x} < q.$''.
\end{proof}
The lemma follows.
\end{proof}
Using Lusin's theorem from measure theory, together with a density argument we get
\begin{theorem}
Any real in $V[G]$ is the image of a continuous function of the ground model defined on a compact $K$ of positive measure such that $g\in K.$
\end{theorem}
\begin{corollary}
Any real in $V[G]$ is included in some nowhere dense closed set of the ground model $V$.
\end{corollary}
\begin{proof}
First of all, there is a $G_\delta$ dense subset of zero measure in $V$, say $X=\bigcap_{n<\omega}U_n,$ so that $g\notin X.$ Hence $g$ belongs to one of the complements, call it $F$.

In order to treat the general case, we use the fact that a real $x$ is the range of $g$ via a continuous function $\Phi$ of $V$, defined on a compact set $K, \mu(K)>0.$ Now $\Phi[K\cap F]$ is a compact nowhere dense set coded in $V$ and contains $\Phi(g)=x.$
\end{proof}
From the Corollary it will follow, once we know Cohen generic reals, that no such real appears in $V[G].$ We close discussing the single random real model by the following. Let $\RR^V$ be the reals of the ground model $V$.
\begin{theorem}
$(a)$ $\RR^V$ is meager,

$(b)$ $\RR^V$ is not measurable.
\end{theorem}

If we consider the effects of adding many random reals, then we  have the following.

\begin{theorem}
($ZFC$) The following are equivalent:

$(a)$ Every $\lusim{\Sigma}_2^1$ set ($PCA$) is Lebesgue measurable,

$(b)$ Almost all reals are random over any inner model $L[\a], \a\in \RR.$
\end{theorem}

\section{Collapsing} The set of conditions $\Col(\aleph_0, \aleph_1)$  is
\begin{center}
$\{p: p$ is a function from a finite subset of $\aleph_0$ into $\aleph_1 \}.$
\end{center}
ordered by reverse inclusion.
\begin{lemma}
In the generic extension, there is an onto map from $\aleph_0 \rightarrow \aleph_1.$ Also other cardinals remain cardinals (because $|\Col(\aleph_0, \aleph_1)|\leq\aleph_1$).
\end{lemma}
\begin{theorem}
($CH$) In $V[G],$ almost all reals are random over $V$.
\end{theorem}
\begin{proof}
The Borel sets of zero measure coded in $V$ form a countable set.
\end{proof}

\section{Amoeba forcing} The set of conditions this time is
\begin{center}
$\{K: K$ is compact $ \subseteq \RR$ and $\mu(K)>1 \},$
\end{center}
ordered by inclusion.
\begin{lemma}
This set satisfies the $c.c.c.$
\end{lemma}
\begin{proof}
Very similar to the case of random forcing.
\end{proof}
\begin{theorem}
The intersection of the compact sets of the generic is a compact set of measure $1$ consisting of random reals.
\end{theorem}
\begin{proof}
We prove it consists of random reals. Let $B$ be a Borel set of zero measure coded in the ground model. $\{K: K\cap B=\emptyset \}$ is dense. This gives the result.

To prove that the measure of the intersection is at least $1$, assume on the contrary it is $\leq 1-\delta.$ Some open set $U$ covers the intersection with $\mu(U) \leq 1-\delta/2,$ and it can be replaced by a finite union of open intervals $U_0$.
Now $\mu(K\setminus U_0) > \delta/2,$ for any $K$ in $G$. Hence $\bigcap\{K\setminus U_0: K\in G  \}\neq \emptyset,$ by compactness.
\end{proof}

\section{The covering forcing} We force with the set of pairs $(k, f)$ such that
\begin{enumerate}
\item $k$ is an integer,
\item $f$ is a function from $\omega$ into the finite subsets of $\omega$ such that $\forall n, |f(n)|\leq n,$ and $|f(n)|$ is bounded.
\end{enumerate}
$(l, g)\leq (k,f)$ iff
\begin{enumerate}
\item $l\geq k,$
\item $g\upharpoonright k= f\upharpoonright k,$
\item $\forall n, g(n) \supseteq f(n).$
\end{enumerate}
It is easily seen that the $c.c.c.$ holds. Let $G$ be generic and  let $\Phi$, the map from $\omega$ into the finite subsets of $\omega$, obtained from the generic set.
\begin{lemma}
$(a)$ $|\Phi(n)|\leq n,$

$(b)$ Any element $\a$ of $\omega^\omega$ of the ground model is eventually covered by $\Phi,$ i.e. $\exists p \forall n\geq p, \a(n)\in \Phi(n).$
\end{lemma}
This is proved by a simple density argument.
\begin{theorem}
In $V[G],$ almost all reals are random over $V$.
\end{theorem}
\begin{proof}
We need a lemma.
\begin{lemma}
$(ZFC)$ Given a set $A$ of measure $0$, there exists a sequence of basic sets (i.e. finite union of open intervals with rational endpoints) $W_n$ such that

$(a)$ $A \subseteq \overline{\lim W_n},$

$(b)$ $\mu(W_n) < 1/2^n.$
\end{lemma}
\begin{proof}
Let $\theta: \omega\times\omega \rightarrow \omega$ be a bijection such that $\theta(p,q)> p,$ except for $p=q=0.$ We then pick up a sequence of open sets $U_p \supseteq A$ with $\mu(U_p) < 1/2^{\theta(p,0)}$. $U_p$ can be written as a disjoint union of intervals which we enumerate as $I_{p, l}.$ We then define by induction on $q$ integers $l_{p,q}$ in such a way that $l_{p,0}=0$ and $\mu(\bigcup_{r\geq l_{p,q}} I_{p,r}) < 1/2^{\theta(p,q)}.$

Let $V_{p,q}$ be $\bigcup \{ I_{p,r}: l_{p,q} \leq r < l_{p, q+1}  \}.$ We get $\mu(V_{p,q}) < 1/2^{\theta(p,q)}.$ So we can slightly extend $V_{p,q}$ in order to get a basic set $\tilde{V}_{p,q}$ satisfying the same inequality.

Clearly any $\a$ in $A$ belongs to some $\tilde{V}_{p,q},$ for fixed $p$; hence to infinitely many of them. We finally let $W_n=\tilde{V}_{p,q}$ if $\theta(p,q)=n.$
\end{proof}
\begin{remark}
By the Borel-Cantelli lemma, it follows that $\overline{\lim W_n}$ has measure $0$.
\end{remark}
We now complete the proof of the theorem. Let $W_{n,i}$ be an enumeration of basic sets of measure $<1/2^n.$ If $\Phi$ is given by the generic, we consider $\bigcup_{i\in \Phi(n)} W_{n,i}.$

Now if $A$ is a Borel set of zero measure, there is by Lemma, an $\a: \omega \rightarrow \omega$ in $V$ such that $A \subseteq \overline{\lim W_{n, \a(n)}}.$ Hence because $\a$ is almost contained in $\Phi,$ we get  $A \subseteq \overline{\lim \bigcup_{i\in \Phi(n)} W_{n,i}}.$

But $\bigcup_{i\in \Phi(n)} W_{n,i}$ has measure $\leq n/2^n.$ Hence $A$ is  included in a fixed zero measure set of $V[G]$.
\end{proof}

\chapter{Cohen forcing}
\section{Adding one Cohen real} Let's start with definition of a new forcing notion.
$\PP$ here is the set of open intervals with rational endpoint. This set of conditions is countable, hence all cardinals of the ground model remain cardinals. Let $G$ be $\PP$-generic over $V$.

\begin{lemma}
There is a single real $g$ which belongs to all intervals $(r,s)$ with $(r,s)\in G.$
\end{lemma}
\begin{proof}
Let $\a=sup\{r: (r,s)\in G \}$ and $\beta=inf\{s: (r,s)\in G       \}.$ First of all note that $\a\leq \beta,$ as otherwise some conditions $(r_1, s_1),(r_2, s_2)$ of $G$ are such that $r_1>s_2.$ This contradicts compatibility. Now if $\a<\b,$ then we pick $q\in \QQ$ such that $\a < q < \beta,$ and we use the dense set $D_q$ defined by
\begin{center}
$D_q=\{(s,t): (s,t) \subseteq (-\infty, q)$ or $(s,t) \subseteq (q, +\infty)    \}.$
\end{center}
Once a condition of $G$ is in $D_q$, it will get $\a, \b < r$ or $r<\a,\b,$ contradiction. Thus $\a=\b,$ which we denote $g$.
\end{proof}
\begin{lemma}
The real $g$ does not belong to any  closed nowhere dense set coded in $V$.
\end{lemma}
\begin{remark}
Such a real is called Cohen generic.
\end{remark}
\begin{proof}
Let $F$ be such a set. If $p$ is given, then $p$ is an open set and $p\setminus F$ is open and $\neq\emptyset,$ hence $q\leq p$ can be found disjoint from $F$. Hence $\exists q\in G, q\cap F=\emptyset.$ The lemma follows.
\end{proof}
Conversely if $g$ is given, then the set of intervals including $g$ can be constructed and denoted by $\tilde{g}$
\begin{theorem}
$\tilde{g}$ is generic iff $g$ does not belong to any nowhere dense closed set of $V$.
\end{theorem}
\begin{proof}
Let $D$ be a dense set; the union of the intervals in $D$ is an open set $X$.
\begin{claim}
It is dense.
\end{claim}
\begin{proof}
Otherwise, some interval $(r,s)$ is disjoint from it, but there is $(r_0,s_0) \subseteq (r,s)$ such that $(r_0,s_0)\in D,$ but then $(r_0,s_0) \subseteq X,$ contradiction.
\end{proof}
Now $g$ belongs to $X$, so it belongs to some interval of $D$.
\end{proof}
It should be noted that if one goes from $G$ to $g$, and back to $\tilde{g},$ we get $G \subseteq \tilde{g},$ hence $G=\tilde{g}.$

\begin{lemma}
Any real of $V[G]$ is the value of $g$ of a Borel measurable function.
\end{lemma}
\begin{proof}
We assume the given real $x$ belongs to $[0,1].$ Let
\begin{center}
$I_0\Vdash$``$\lusim{x}$ is a real of $[0,1]$''.
\end{center}
Then for any rational number $q$ we consider $\{I: I\Vdash$``$\lusim{x} < q$''$ \}.$ Taking the union of these conditions $I$ yields an open set $U_q$. We let $\Phi_q$ to be $q$ on $U_q$ and $1$ otherwise. the required function is $\Phi=\inf_{q\in \QQ}\Phi_q.$
\begin{claim}
The value of $\Phi$ at $g$ is $x.$
\end{claim}
\begin{proof}
It is easy to show that $x\geq \Phi(g).$ In the other direction, if $x > \Phi(g),$ then for some $q$, $x > \Phi_q(g)$. This means $g\in U_q$, hence for some $I$, $g\in I, I\Vdash$``$\lusim{x} < q$'', contradiction, because $\Phi_q(g)=q.$
\end{proof}
The lemma follows.
\end{proof}
\begin{corollary}
Any real is the value at $g$ of a continuous function of $V$ defined on a dense $G_\delta$ subset of $\RR$
\end{corollary}
\begin{proof}
This is because a Borel measurable function can be restricted to some dense $G_\delta$ subset $X$ so as to become continuous on $X$.
\end{proof}

Other properties of the model are given in the next lemma.
\begin{lemma}
$(a)$ $\RR^V$ does not have the Baire property.

$(b)$ $\RR^V$ is of zero measure.
\end{lemma}
Note that $(a)$ implies that no real in $\RR^V$ is random: this is because a single random real makes $\RR^V$ meager.

\section{Adding Cohen reals side by side}
Let $\k$ be a cardinal $>\omega.$ We force with the set of functions $p$ with finite domain $\subseteq \k\times\omega$ into $\{0,1 \}.$
\begin{lemma}
This set has the $c.c.c.$
\end{lemma}
This is a consequence of the so called $\Delta-$lemma, which is a valuable tool in establishing $c.c.c.$
\begin{lemma}
($\Delta-$lemma) Let $\mathcal{W}$ be an uncountable collection of finite sets. there is an uncountable $\mathcal{Z} \subseteq \mathcal{W}$ and a finite set $S$, such that
\begin{center}
$\forall X, Y\in \mathcal{Z}, X\neq Y \Rightarrow X\cap Y=S.$
\end{center}
\end{lemma}
\begin{proof}
Let $\mathcal{W}$ be an uncountable collection of finite sets. We may assume that for some $n$ we have $\forall X\in \mathcal{W}, |X|=n.$ Then the lemma is proved by induction on $n$. the lemma is trivial for $n=1.$ Assume $n=m+1,$ and the lemma holds for $m$.

{\bf Case 1.} Some element $a$ belongs to uncountably many $X$'s; we restrict the attention to the set $\mathcal{W}_0=\{X\setminus\{a\}: X\in \mathcal{W}$ and $a\in X  \}$ and apply the induction hypothesis.

{\bf Case 2.} Each $a$ belongs to countably many $X$'s. Then there is a disjoint family $ (X_\a: \a<\aleph_1)$ constructed as follows:
the $X_\a, \a<\b$ have countably many elements, hence some element $Y$ is such that $\forall \a<\b, Y\cap X_\a=\emptyset.$ We define this as  $X_\b.$
\end{proof}
We now turn to the proof of Lemma 3.2.1.
\begin{proof}
If an uncountable antichain  $(p_\xi: \xi<\aleph_1 )$ exists, the domain can be made to satisfy the conclusion of the $\Delta-$lemma. Now the value $S$ of $dom(p_\xi)\cap dom(p_\zeta)$ is fixed and $p_\xi \upharpoonright S$ values in a countable set, extracting one more time, we may assume $p_\xi \upharpoonright S$ is constant. But then any two conditions are compatible.
\end{proof}
\begin{theorem}
the family
\begin{center}
$a_\xi= \sum_{f(\xi,n)=1, n\geq 1}1/2^n$
\end{center}
is a set of distinct Cohen generic reals.
\end{theorem}
\begin{proof}
We first prove each $a_\xi$ is generic. Let $F$ be a closed nowhere dense set of the ground model, and let
$D=\{p:$ for some $n, p(\xi,1), ..., p(\xi,n)$ are defined and $s= \sum_{p(\xi,i)=1, i\leq n}1/2^i$
   and $t=s+1/2^n$ are such that $[s,t]\cap F=\emptyset \}.$

We claim that this set is dense. This follows from the fact that $F$ is nowhere dense and is just technical. We then note that if $p\in G\cap D,$ then $a_\xi\in [s,t];$ so that $a_\xi\notin F.$

In order to show that the $a_\xi$'s are distinct, then as they are not rational, we have only to exhibit distinct dyadic developments. Now if $\xi\neq \zeta,$ it is easily seen that
\begin{center}
$\{p: \exists n, p(\xi,n)\neq p(\zeta,n) \}$
\end{center}
is dense; the required result follows.
\end{proof}
In particular if we take $\k=\aleph_2,$ we get a model where $CH$ fails.
 We also have the following
\begin{theorem}
($ZFC$) The following are equivalent:

$(a)$ Every $\lusim{\Sigma}_2^1$ set has the Baire property,

$(b)$ The set of reals Cohen generic over any $L[\a]$ is comeager.
\end{theorem}
%\begin{proof}
%$(a)\Rightarrow (b)$ uses the Baire category analogue of Fubini's theorem.

%$(b)\Rightarrow (a)$ is as follows: We pick $\a$ a code for a Borel set $B$ from which the $\lusim{\Sigma}_2^1$ set is defined by the formula
%\begin{center}
%$\theta(z)=\exists x \forall y \phi(x,y,z).$
%\end{center}
%Then we consider $X=\{ I: I\Vdash$``$\theta(\tau)$''$\}$, where $\tau$ is a name for the generic real. The union of all these intervals is an open set $U$ and we have for %$g$ generic over $L[\a],$
%\begin{center}
%$g\in U  \Leftrightarrow \theta(\tau)$ (as in the measure case).
%\end{center}
%This finishes the proof.
%\end{proof}
%The condition ``the generic reals over $V$ form a comeager set'' can be realized by collapsing $\aleph_1$, provided $V$ satisfies $CH$. Other %techniques are known.

We stop for a while in the connection between $2^\omega$ and $[0,1];$ the continuous map
\begin{center}
$\theta: \a \mapsto \Sigma_{n=0}^{\infty} \a(n)/ 2^{n+1}$
\end{center}
has the following properties:
\begin{enumerate}
\item The range of a closed nowhere dense set is a closed nowhere dense set,
\item The inverse image of a closed nowhere dense set is also a closed nowhere dense set.
\end{enumerate}
From this it follows that if $\a$ belongs to no nowhere dense set of $V$, then $\theta(\a)$ is generic.

So finally if the union of all closed nowhere dense sets of $2^\omega$ with a code in $V$ is meager in $2^\omega;$ then the same happens in $[0,1]$ and therefore the set of generic reals over $V$ is comeager.

We now consider the set of non-empty closed nowhere dense sets of $2^\omega$ and force with conditions which are pairs $(k, F)$ where $k$ is an integer and $F$ is a nowhere dense closed set. If $F$ is such a set, the tree $T_F$ of $F$ is defined as follows: If $s\in 2^{<\omega},$ we let
\begin{center}
$\hat{s}=\{ \a\in 2^\omega : \a$ extends $s \}.$
\end{center}
Then $T_F=\{s: \hat{s}\cap F \neq \emptyset \}$. $(l, G) \leq(k, F)$ iff $l\geq k$ and $T_F \cap 2^k = T_G \cap 2^k.$
\begin{lemma}
The set of conditions satisfies the $c.c.c.$
\end{lemma}
\begin{proof}
This is because conditions $(k, F), (k, G)$ such that  $T_F \cap 2^k = T_G \cap 2^k$ are compatible (common extension is $(k, F\cup G)$).
\end{proof}
If a generic set $g$ is given,  we consider the tree
\begin{center}
$T=\{s: \exists (k, F)\in g, |s|\leq k$ and $ s\in T_F  \}$
\end{center}
and the closed set
\begin{center}
$\Phi=\{\a: \forall n, \a \upharpoonright n \in T    \}.$
\end{center}
\begin{claim}
$\Phi$ defines a nowhere dense closed set.
\end{claim}
\begin{proof}
This is because if $s$ is given, then the set of conditions
\begin{center}
$\{(k, F):$ for some extension $t$ of $s$ of length $k, t\notin T_F \}$
\end{center}
is dense.
\end{proof}
Now the result that the generic reals are comeager is achieved by the following.
\begin{lemma}
Any nowhere dense set in $V$ is covered by a finite union of translations of $\Phi$.
\end{lemma}
The translations are defined from finite subsets $u$ of $\omega$ by
\begin{center}
 $T_u(\alpha)=\b$ $\hspace{.5cm}$ iff $\hspace{.5cm}$ $\left\{
\begin{array}{l}
        \a(n)=\b(n) \hspace{2.2cm} \text{ if } n\notin u,\\
        \a(n)=1-\b(n) \hspace{1.65cm} \text{ if } n\in u.
     \end{array} \right.$
\end{center}
They are continuous automorphisms of $2^\omega$.
\begin{proof}
We let $F_0$ be a non-empty nowhere dense set of $V$. Given a condition $(k, F),$ we define a new condition $(k, F'),$ where $\a\in F'$ iff $\a \upharpoonright k \in T_F, \a\in F$ or $\a\in T_u(F_0)$ for some $u \subseteq \{0, ..., k \}.$

This is a closed nowhere dense set and $(k, F')$ is an extension of $(k, F).$ Now given $\b\in F_0,$ we can define $u \subseteq \{0, ..., k \}$ such that $T_u(\b)$ is in $F'$, i.e. $\b\in T_u(F').$ Therefore
\begin{center}
$F_0 \subseteq \bigcup_{u \subseteq \{0,...,k \}} T_u(F').$
\end{center}
Finally we have shown that the set of conditions $(k, F')$ such that $F_0 \subseteq \bigcup_{u \subseteq \{0,...,k \}} T_u(F')$ is dense. The result follows.
\end{proof}

\chapter{Sacks forcing}
If $g$ is Cohen generic over $V$, then there are $A,B \subseteq \omega, A,B\in V[g]$
such that $A$ and $B$ are independent, in the sense that $A\notin V[B]$ and $B\notin V[A].$ Take $A$ to code up $g \upharpoonright \{2n: n<\omega   \}$, $B$ to code up $g \upharpoonright \{2n+1: n<\omega   \}.$ Sacks found a way to add a generic $s: \omega\rightarrow 2$ to $V$ so that the above doesn't happen: if $A,B\in V[s]$ $(A, B \subseteq \omega)$ and $A\notin V,$ then $B\in V[A],$ thus if $V=L,$ then $L[s]\models$``$ZFC+$ there are exactly two degrees of constructibility'', where for $A\subseteq \omega,$ the constructibility degree of $A=\{B\subseteq \omega: A\in L[B]$ and $B\in L[A]   \}.$ We will consider in this chapter this result and other facts about Sacks forcing.

\section{Sacks reals}

For $u,v\in 2^{<\omega},$ let $u \leq v$ if $u$ is an initial segment of $v$, $u<v$ if $u$ is a proper initial segment of $v$, $u \nsim v$ if $u \nleq v$ and $v \nleq u.$ A perfect subtree of $2^{<\omega}$ is a nonempty $T \subseteq 2^{<\omega}$ which is downward closed ($u\in T, v\leq u \Rightarrow v\in T$) and splits above each node ($u\in T \Rightarrow \exists v, v', u < v, v', v \nsim v').$ Let $Lev_n(T)$ be the set of nodes on the $n$-th level of $T$. Let $stem(T)=\{ u\in T: \forall v\in T (v <u  \Rightarrow v$ has only one immediate successor in $T)\}.$ For $t\in T, T_t=\{u\in T: u\leq t$ or $t\leq u\}.$

Then $\SSS,$ the partial ordering for adding a Sacks real $s: \omega \rightarrow 2$, is $\{T: T$ is a perfect subtree of $2^{<\omega}       \}$, ordered by inclusion. Then if $G$ is $\SSS$-generic over $V$, define $s=\bigcup_{T\in G}stem(T);$ by genericity it is easy to see that $s: \omega \rightarrow 2$, say $s$ is the Sacks real associated to $G$.
\begin{lemma}
$V[G]=V[s].$
\end{lemma}
\begin{proof}
Clearly $s\in V[G].$ To see $G\in V[s],$ we claim
\begin{center}
$G=\{ T\in \SSS: s$ is a branch through $T     \}.$
\end{center}
The $\subseteq$ is not hard to check. Now suppose $T\in \SSS,$ and $U\Vdash$``$\lusim{s}$ is a branch through $T$''.
Then we claim $U \subseteq T$ (for if not, we could extend $U$ in such a way as to force that $s$ is not a branch through $T$). Whence $U\Vdash$``$T\in \Gamma$'' (where $\Gamma$ is the canonical name for $G$) and we are done.
\end{proof}
\begin{lemma}
$\SSS$ is weakly $(\omega, \infty)$-distributive, i.e. if $T\Vdash$``$\tau: \omega \rightarrow V,$ then $\exists (F_n: n<\omega) \in V,$ each $F_n$ finite and $U \subseteq T$ such that  $U\Vdash$``$\forall n<\omega,  \tau(n)\in F_n$''.
\end{lemma}
We will prove the lemma by a ``fusion argument'', that we now explain, and below will refer back to it without details. For $T, U\in \SSS$ and $k<\omega$ let $U \leq_k T$ if $U \subseteq T$ and $U_{<k}=T_{<k},$ where $T_{<k}=\bigcup_{n<k}Lev_n(T).$
\\
{\bf Fusion Lemma:} Suppose that $T(0) \geq T(1) \geq ...$ is a decreasing sequence of conditions in $\SSS$ and $k_0 < k_1 < ... <\omega$ are such that $T(n+1) \leq_{k_n} T(n),$ and such that for each $t\in Lev_{k_n}(T(n)),$ there are $u,v>t, u\nsim v$ and $u, v\in T(n+1)_{< k_{n+1}}$ Then $T(\omega)=\bigcap_n T(n)$ is a perfect tree extending each $T(n).$

We now turn to the proof of Lemma 4.1.2.
\begin{proof}
Let $T(0)=T,$ pick $k_0 < \omega$ be arbitrary, and for each $t\in Lev_{k_0}(T(0))$ pick $S(t) \leq T_t$ and $v_t\in V$ such that $S(t)\Vdash$``$\tau(0)=v_t$''. Let $F(0)=\{v_t: t\in Lev_{k_0}(T(0))     \}$ and $T(1)=\bigcup \{S(t): t\in Lev_{k_0}(T(0))\}.$ Let $k_1 > k_0$ be such that the splitting condition for each $t\in Lev_{k_0}(T(0))$ is satisfied. Now construct $T(2), T(3), ...$ similarly. Then $T(\omega) \Vdash$``$\forall n<\omega, \tau(n)\in F_n$'', since every extension of $T(\omega)$ must be compatible, for each $n$, with one of the $S(t)$'s defined at stage $n$.
\end{proof}
\begin{lemma}
For $s$ Sacks generic over $V$, $\omega_1^{V[s]}=\omega_1$ and if $CH$ holds in $V$, then $\Card^{V[s]}=\Card^V$.
\end{lemma}
\begin{proof}
$\omega_1^{V[s]}=\omega_1$ follows from Lemma 4.1.2. If $CH$ holds in $V$, then since $|\SSS|=2^{\aleph_0}, \SSS$ has the $\aleph_2-c.c,$ so cardinals above $\omega_1$ are preserved from $V$ to $V[s]$ as well.
\end{proof}
Note that if $W$ is a generic extension of $V, A \subseteq \omega, A\in W\setminus V,$ then there is an infinite $B\subseteq \omega, B\in W$ such that no infinite subset of $B$ lies in $V$. To see this, take a bijection $ f: [\omega]^{<\omega} \leftrightarrow \omega, f\in V$ and let $B=\{f(A\cap n): n<\omega      \}.$ For $W=V[s], s$ a Sacks real, the next best thing happens: there is an infinite $C\subseteq \omega, C\in V$ with $C\subseteq A$ or $C \subseteq \omega\setminus A.$ Written in terms of functions, this is
\begin{lemma}
If $s$ is Sacks generic over $V$,  then every $f:\omega \rightarrow 2, f\in V[s]$ has an infinite subset belonging to $V$.
\end{lemma}
\begin{proof}
Otherwise some $T\Vdash$``$\lusim{f}: \omega \rightarrow 2$ has no infinite subset in $V$''. For $U\in \SSS,$ say that $U$ decides $\lusim{f}(n)$ ($U \| \lusim{f}(n)$) if for some $i<2, U\Vdash$``$\lusim{f}(n)=i$''. By the assumption on $T$, for every $U\leq T, \{n: U \parallel \lusim{f}(n)   \}$ is finite. Now do a fusion argument to construct a sequence $T=T(0) \geq_{k_0} T(1) \geq_{k_1} ...$ and $k_0 < k_1 < ... <\omega.$ At stage $n$, let $\mathcal{V}=\{T(n)_t: t\in Lev_{k_n}(T(n))   \}.$ By the above finiteness assumption, there is an $m_n$ such that no $U\in \mathcal{V}$ decides $\lusim{f}(m_n)$, so for each such $U$, pick an $S_U \leq U, S_U\Vdash$``$\lusim{f}(m_n)=0$''. Then let $T(n+1)=\bigcup_{U\in\mathcal{V}}S_U,$ and pick $k_{n+1}$ as in 4.2.1. Letting $T(\omega)=\bigcap_n T(n),$ and $h(m_n)=0,$ all $n<\omega,$ then $T(\omega) \Vdash$``$h \subseteq \lusim{f}$''.
\end{proof}
We now show that Sacks forcing leads to a minimal generic extension.
\begin{theorem}
Suppose $s$ is Sacks generic over $V$. If $A,B \in V[s], A,B \subseteq \omega$ and $B\notin V,$ then $A\in V[B].$
\end{theorem}
\begin{proof}
It suffices to show $s\in V[B].$ Suppose $T\Vdash$``$\lusim{s} \notin V[\lusim{B}]$ and $\lusim{B} \notin V$''. We will construct a fusion sequence $T=T(0) \geq_{k_0} T(1) \geq_{k_1} ...$ such that letting $T(\omega)=\bigcap_n T(n), T(\omega)$ will have the following property: if $t\in T(\omega)$ is a Sacks node (i.e. $t^{\frown}0, t^{\frown}1 \in T(\omega)$), then there is a $m$ such that either
\begin{center}
$T(\omega)_{t^{\frown}0}\Vdash$``$m\in \lusim{B}$'' and $T(\omega)_{t^{\frown}1}\Vdash$``$m\notin \lusim{B}$''
\end{center}
or
\begin{center}
$T(\omega)_{t^{\frown}0}\Vdash$``$m\notin \lusim{B}$'' and $T(\omega)_{t^{\frown}1}\Vdash$``$m\in \lusim{B}$''.
\end{center}
Furthermore the function $t \mapsto m$ is in $V$, so assuming $T(\omega)$ is in the Sacks generic set, $s$ can be reconstructed from $B$, that is $T(\omega)\Vdash$``$\lusim{s}\in V[\lusim{B}]$'', a contradiction.
\end{proof}

\section{Adding many Sacks reals}
A number of independence proofs require one to add $\k$-many Sacks reals to $V$ rather than just one. If $\k$ is a cardinal (finite or infinite), $\SSS_\k,$ the partial ordering for adding $\k$-many      Sacks reals, is the set of all $f: \k \rightarrow \SSS,$ such that $\{\a<\k: f(\a) \neq 2^{<\omega}  \}$ is countable (where $\SSS$ is the Sacks forcing). Order $\SSS_\k$ by $f\leq g \Leftrightarrow \forall \a<\k, f(\a) \leq g(\a).$ Thus if $\k \leq \omega, \SSS_\k$ is just the $\k$-fold direct product of $\SSS.$ We consider which of the above results generalize to $\SSS_\k.$
\begin{lemma}
The analogues of Lemmas 4.1.1, 4.1.2 and 4.1.3 for $\SSS_\k$ hold.
\end{lemma}
\begin{proof}
$(a):$ If $G$ is $\SSS$-generic over $V$, $\a < \k,$ let $s_\a = \bigcup_{f\in G}stem(f(\a)).$ Then each $s_\a$ is Sacks generic and $V[G]=V[(s_\a: \a < \k)]$ as before.

$(b):$ To prove the weak distributivity of $\SSS_\k,$ we need the following version of the fusion lemma. For $f\in \SSS_\k,$ the support of $f$ is the countable set $\supp(f)=\{\a<\k: f(\a) \neq 2^{<\omega} \}.$
\\
{\bf Generalized fusion lemma:} suppose $f(0), f(1), ...,$ $\a_0, \a_1, ...$ and $k_0 < k_1 < ... <\omega$ are such that
\begin{enumerate}
\item $f(n+1) \leq f(n),$
\item for each $\a\in \{\a_0, ..., \a_n \}, f(n+1)(\a)\cap 2^{<k_n}=f(n)(\a)\cap 2^{<k_n}$,
\item for each $\a\in \{\a_0, ..., \a_n \}$, each $t\in (f(n)(\a))_{k_n}$ there are $u, v > t$ such that $u \nsim v$ and $u,v\in (f(n+1)(\a))_{< k_{n+1}}$,
\item $\{ \a_n: n<\omega \} = \bigcup_n \supp(f(n)).$
\end{enumerate}
Then $f(\omega): \k \rightarrow \SSS$ defined  by $f(\omega)(\a)=\bigcap_n f(n)(\a)$ is a member of $\SSS_\k.$

Now given $f\Vdash$``$\lusim{g}: \omega \rightarrow V$'', construct a  fusion sequence $f=f(0) \geq f(1) \geq ...$ reducing the possible values of $\lusim{g}(n)$ to a finite set $F_n$ at the $n$-th stage of the fusion sequence, and simultaneously choosing $\{\a_0, \a_1, ...\}$ so that $4$ holds at the end.

$(c):$ $\omega_1^{V[(s_\a:\a<\k)]} =\omega_1$ again follows from $(b)$ above. A $\Delta$-system argument, assuming $2^{\aleph_0}=\aleph_1$ in $V$, gives that $\SSS_\k$ has the $\aleph_2-c.c.$
\end{proof}
\begin{lemma}
The analogue of Lemma 4.1.5 fails for $\SSS_\k.$
\end{lemma}
\begin{proof}
By genericity $s_\a \notin V[s_\beta]$ for $\a\neq \beta.$ It is true though that if $A \subseteq \omega, A\in V[(s_\a: \a<\k)]\setminus V,$ then there is some $\a<\k$ with $s_\a\in V[A].$
\end{proof}
The reminder of this section is about the analogue of Lemma 4.1.4 for $\SSS_\k$ and related results.
\begin{theorem}
If $A \subseteq \omega, A\in V[(s_\a: \a<\k)],$ then there is an infinite $B\subseteq \omega, B\in V$ with $B \subseteq A$ or $B \subseteq \omega\setminus A.$
\end{theorem}
\begin{proof}
We first consider the case $\k=d<\omega.$ Given a condition $(T_i: i<d) \in \SSS_d,$ and a term $\lusim{A}$ for  subset of $\omega,$ a fusion argument gives $T'_i \leq T_i (i<d)$ and an infinite $C \subseteq \omega$ such that for each $n\in C$ there is an $h_n: \bigotimes_{i<d}Lev_n(T'_i) \rightarrow 2,$ such that for each $\vec{t}=(t_0, ..., t_{d-1})\in \bigotimes_{i<d}Lev_n(T'_i)$
\begin{center}
$((T'_0)_{t_0}, ..., (T'_{d-1})_{t_{d-1}}) \Vdash$``$n\in \lusim{A}$''$\Leftrightarrow h_n(t_0, ..., t_{d-1})=1.$
\end{center}
What we would like is an $l<2,$ an infinite $C' \subseteq C$ and $T''_i \leq T'_i (i<d)$ such that for each  $\vec{t}\in \bigotimes_{i<d}Lev_n(T''_i),$ when $n\in C', h_n(\vec{t})=l.$ A version of a combinatorial theorem of Halpern and Lauchli gives this fact; it was originally proved by them for a different application, and has had a number of other uses. If  $\vec{T}=(T_0, ..., T_{d-1})\in \SSS_d$ and  $C \subseteq \omega,$ define $\bigotimes^C \vec{T}$ to be $\bigcup_{n\in C}\bigotimes (Lev_n(T_0), ..., Lev_n(T_{d-1})).$

 \begin{lemma} (Perfect tree version of the Halpern-Lauchli theorem) If $\vec{T} \in \SSS_d, C \subseteq \omega$ is infinite,  $\bigotimes^C \vec{T}=K_0 \cup K_1,$ then there are $T'_i \leq T_i, \vec{T'}=(T'_0, ..., T'_{d-1})\in \SSS_d,$ an $l<2$ and an infinite $C' \subseteq C$ with  $\bigotimes^{C'} \vec{T'} \subseteq K_l.$
\end{lemma}
We state a stronger form of the lemma. If $\vec{T}=(T_0, ..., T_{d-1}) \in \SSS_d, n<\omega,$ an $n$-dense sequence is an $A_0, ..., A_{d-1}$ such that for some $m\geq n,$ each $A_i \subseteq Lev_m(T_i)$ and for each $i, \forall t\in Lev_n (T_i) \exists u\in A_i, t\leq u.$
For $\vec{t}\in \bigotimes_{i<d}T_i$, an $m$-sequence above $\vec{t}$ is an $m$-dense sequence in $((T_0)_{t_0}, ..., (T_{d-1})_{t_{d-1}}).$
\begin{lemma}
(Dense sequence version of the Halpern-Hauchli theorem) If $d<\omega, \vec{T}\in \SSS_d, C \subseteq \omega$ is infinite and   $\bigotimes^C \vec{T}=K_0 \cup K_1,$ then either
\begin{enumerate}
\item $\forall n \exists n$-dense sequence $A_0, ..., A_{d-1}$ with $\bigotimes \vec{A} \subseteq K_0,$ or
\item $\exists \vec{t}\in \bigotimes \vec{T} \forall n \exists n$-dense sequence  $A_0, ..., A_{d-1}$ above $\vec{t}$ with $\bigotimes \vec{A} \subseteq K_1$.
\end{enumerate}
\end{lemma}
\begin{proof}
Let $\k=\beth_{2d-1}(\aleph_0)^+.$  Let $\PP$ be the partial ordering for adding $\k$-many Cohen generic branches $b_{i,\a} (i<d, \a<\k)$ through each $T_i$. Thus
\begin{center}
$p\in \PP \Leftrightarrow p=(p_i: i<d),$ where $\dom(p_i) \subseteq [\k]^{<\aleph_0}, \range(p_i) \subseteq T_i.$
\end{center}
Then define
\begin{center}
$p\leq q \Leftrightarrow \forall i<d (\dom(p_i)\supseteq \dom(q_i)$ and for each $\a\in \dom(q_i), p_i(\a) \geq_{T_i} q_i(\a)).$
\end{center}
Note that two conditions $p,q$ are compatible if
$\forall i<d \forall \a\in \dom(p_i)\cap \dom(q_i) ( p_i(\a) \leq q_i(\a)$ or $ q_i(\a) \leq p_i(\a)).$ We will use the machinery of forcing with $\PP$ rather than actually taking a generic extension; we will informally use the notation $V[G]$ for the imaginary generic extension. Let $\lusim{U}$ be a name in the language of $\PP$ such that
\begin{center}
$\Vdash_{\PP}$``$\lusim{U}$ is a non-principal ultrafilter on $\omega$ with $C\in \lusim{U}$''.
\end{center}
Recall that in $V[G], b_{i,\a}$ is the $\a$-th generic branch through $T_i$: $b_{i,\a}=\{ t\in T_i: \exists p\in G, p_i(\a)=t \}.$ In $V[G],$ define for $\a_0 < \a_1 < ... < \a_{d-1} < \k$
\begin{center}
$\{\a_0, .., \a_{d-1}\} \in \tilde{K}_l \Leftrightarrow \{n: (b_{0,\a_0}, ..., b_{d-1, \a_{d-1}})\in K_l \}\in U.$
\end{center}
Back in $V$, pick for each $\a_0 < \a_1 < ... < \a_{d-1}$ a $p_{\vec{\a}}\in \PP$ and an $l_{\vec{\a}} <2$ such that $p_{\vec{\a}}\Vdash$``$\{\a_0, .., \a_{d-1}\} \in \tilde{K}_{l_{\vec{\a}}}$''  as follows:
If $\Vdash_{\PP}$``$\{\a_0, .., \a_{d-1}\} \in \tilde{K}_0$'', let $(p_{\vec{\a}})_i=\{ (\a_i, stem(T_i))\}$ for each $i<d.$ Otherwise pick $p_{\vec{\a}}$ arbitrary forcing $\{\a_0, .., \a_{d-1}\} \in \tilde{K}_1 $, where by extending we may assume $\a_i \in \dom(p_{\vec{\a}})_i$ for each $i<d.$ Define the type of $p_{\vec{\a}}$ to be $( (p_{\vec{\a}})_i({\a_i}): i<d )$, a member of $\bigotimes\vec{T}.$

If $\a_0 < ... < \a_{2d-1},$ let $Y_{\vec{\a}}=\{\a_0, \a_1\}\otimes \{ \a_2, \a_3 \}\otimes ... \otimes \{\a_{2d-2},\a_{2d-1} \}$. If for some $\vec{\gamma}, \vec{\delta}\in Y_{\vec{\a}}, p_{\vec{\gamma}} \nparallel p_{\vec{\delta}},$ let $W(\vec{\a})$ be a witness of this fact (for example, if $\vec{\gamma}=\{\a_0, \a_2, \a_4, ...  \}$ and $\vec{\delta}=\{\a_1, \a_3, \a_5, ... \}, W(\vec{\a})$ could be taken as $(i,j,k,t,u)$, where $t,u\in T_i, t$ incompatible with $u$, and for some ordinal  $\theta, \theta=$the $j$-th member of $\dom(p_{\vec{\gamma}})_i=$the $k$-th member of $\dom(p_{\vec{\delta}})_i$ and $(p_{\vec{\gamma}})_i(\theta)=t$ and $(p_{\vec{\delta}})_i(\theta)=u$).
 If for all $\vec{\gamma}, \vec{\delta}\in Y_{\vec{\a}}, p_{\vec{\gamma}} \| p_{\vec{\delta}},$ let $W(\vec{\a})=\emptyset.$

Color $[\k]^{2d}$ by $c(\{\a_0, ..., \a_{2d-1}  \})= ($the $l$ with $p_{\a_0, ..., \a_{2d-1}}\Vdash$``$\{\a_0, ..., \a_{d-1} \} \in \tilde{K}_l$'', type $p_{\a_0, ..., \a_{d-1}}, W(\a_0, ..., \a_{2d-1})).$ Then $|\range(c)|\leq \aleph_0,$ so by Erdos-Rado theorem choose an infinite $B \subseteq \k$ such that $c$ is homogeneous on $[B]^{2d}.$ Then
\begin{enumerate}
\item there is $l<2$ with $p_{\vec{\a}}\Vdash$``$\{\a_0, ..., \a_{d-1}\}\in \tilde{K}_l$'', all $\vec{\a}$ from $B$.
\item there is $(t_0, ..., t_{d-1})\in \bigotimes\vec{T}$ with $(p_{\vec{\a}})_i(\a_i)=t_i,$ all $\a_0, ..., \a_{d-1}$ from $B$.
\end{enumerate}
\begin{claim}
If $\a_0 <  ... < \a_{2d-1}$ from $B$ and $\vec{\gamma}, \vec{\delta}\in Y_{\vec{\a}},$ then $p_{\vec{\gamma}} \| p_{\vec{\delta}}.$
\end{claim}
\begin{proof}
Otherwise for all $\vec{\a}$ from $B, W(\vec{\a})$ is the same witness to incompatibility. Assume the incompatibility is $p_{\a_0, \a_2}, ... \nparallel p_{\a_1, \a_3}, ...$ via $(i,j,k,t,u)$ as in the example above; the other patterns are handled similarly. Pick a sequence $\a_0 < \b_0 < \gamma_0 < \a_1 < \b_1 < \gamma_1 < ... < \a_{d-1} < \b_{d-1} < \gamma_{d-1}$ from $B$. Using $W(\a_0, \b_0, \a_1, ...)=W(\a_0, \gamma_0, \a_1, ...)=W(\b_0, \gamma_0, \b_1, ...)$ obtain $\theta=$ the $j$-th member of $\dom(p_{\vec{\a}})_i=$ the $k$-th member of $\dom(p_{\vec{\b}})_i=$ the $k$-th member of $\dom(p_{\vec{\gamma}})_i,$ but also $(p_{\vec{\b}})_i(\theta)=t$ and $(p_{\vec{\b}})_i(\theta)=u,$ a contradiction.
\end{proof}
We are now ready to complete the proof. Given the $\vec{t}$ from $(2)$ and an $n<\omega,$ we want to find an $n$-dense sequence $\vec{F}$ above $\vec{t}$ with $\bigotimes\vec{F} \subseteq K_l$ ($l$ as in $(1)$). Let $N_i=|Lev_n((T_i)_{t_i})|,$ pick $H_i \subseteq B$ of size $N_i$ ($i<d$) with $\a\in H_i, \b\in H_j \Rightarrow \a<\b$ ($i<j<d$). Let $Z=\{(\a_0, ..., \a_{d-1}): \forall i<d, \gamma_i\in H_i   \}.$ Then if $\vec{\gamma}, \vec{\delta} \in Z,$ then $p_{\vec{\gamma}} \| p_{\vec{\delta}}.$ Let $p$ extends all $p_{\vec{\gamma}}, \vec{\gamma}\in Z.$ Extend $p$ to $\tilde{p}$ such that for all $i<d, \tilde{p}_i \upharpoonright H_i$ is $1-1$ onto $Lev_n((T_i)_{t_i}).$ Now $p_{\vec{\gamma}}\Vdash$``$b_{\vec{\gamma}}\in \tilde{K}_l$'', i.e. $\lusim{V}_{\vec{\gamma}}=\{n: b_{\vec{\gamma}}(n)\in K_l\}\in \lusim{U}.$ Extend $\tilde{p}$ to
$\tilde{\tilde{p}}$ such that for some $m$, $\tilde{\tilde{p}}\Vdash$``$m\in \bigcap_{\vec{\gamma}}V_{\vec{\gamma}}$''. We may assume by extending further that for each $i$ and $\delta\in H_i$ there is a $t_\delta \in Lev_m((T_i)_{t_i})$ with $\tilde{\tilde{p}}\Vdash$``$t_\delta\in b_{i,\delta}$''. Let $F_i=\{ t_\delta: \delta\in H_i\}.$ Then $(F_i: i<d)$ is an $m$-dense set above $\vec{t}$ with $\bigotimes\vec{F} \subseteq K_l,$ as required.
\end{proof}
This gives a proof of 4.2.3 for $\SSS_d, d<\omega.$ For the case of $\SSS_\k, \k$ infinite, it is not hard to see by a fusion argument that it suffices to show for $\k=\omega.$ For this we need the following.
\begin{lemma}
($\omega$-dimensional version of the Halpern-Lauchli theorem) If $\vec{T}\in \SSS_\omega, C \subseteq \omega$ is infinite and $\bigotimes^C\vec{T}=K_0 \cup K_1,$ then $\exists l<2 \exists C' \subseteq C$ infinite $\exists T'_i\leq T_i$ with $\bigotimes^{C'}\vec{T'} \subseteq K_l.$
\end{lemma}
As for a dense set version of this lemma, one can get a result giving either dense sequence in color class $K_0$ or perfect subtree in color class $K_1$.
\end{proof}

\chapter{Namba forcing}
In this chapter we present, under $CH$, a forcing construction of Namba, which changes the cofinality of $\aleph_2$ into $\aleph_0$ without adding any new reals (and hence without collapsing $\aleph_1$).
\section{Changing cofinality of $\aleph_2$ into $\aleph_0$ without adding new reals}
Let's start with the definition of forcing conditions. The Namba forcing $\NN\MM$ consists of pairs $(t, T),$ where
\begin{enumerate}
\item $T \subseteq \omega_2^{<\omega}$ is a tree, i.e., it is closed under initial sequences,
\item $t$ is the stem of $T$, i.e., for all $s\in T, s \upharpoonright |t|=t$ and $|\Suc_T(t)| > 1,$ where $\Suc_T(t)=\{t^{\frown} \langle \a \rangle:    t^{\frown} \langle \a \rangle\in T   \},$
    \item For each $s\in T$ there is $s' \geq_T s$ such that $|\Suc_T(s')|=\aleph_2.$
\end{enumerate}
Note that a Namba tree $T$ can be pruned so as to get that  $|\Suc_T(s)|\in \{1, \aleph_2\},$ for each $s\in T.$ Thus we will always assume that Namba trees are of this form. For a tree $T$ and $t\in T,$ set $T_t=\{s\in T: s\leq_T t$ or $ t\leq_T s    \}.$

Namba forcing is equipped with the partial order $(s, S) \leq (t, T)$ iff $s\in T$ and $S \subseteq T_s.$
\begin{lemma}
Let $G$ be $\NN\MM$-generic over $V$. Then $cf^{V[G]}(\aleph_2)=\aleph_0.$
\end{lemma}
\begin{proof}
Let $C=\bigcup\{t: \exists T, (t, T)\in G\}.$ It is easily seen that $C$ is an $\omega$-sequence cofinal in $\aleph_2.$
\end{proof}
\begin{lemma}
$(CH)$ The forcing $\NN\MM$ adds no  new reals.
\end{lemma}
\begin{proof}
Let $(t, T)\in \NN\MM,$ and let $\lusim{a}$ be a $\NN\MM$-name of a real, i.e., $\Vdash_{\NN\MM}$``$\lusim{a}: \omega \rightarrow \omega$''. We will construct a stronger condition $(t, T^*)\leq (t, T)$ such that for each $n<\omega$ and each $\eta\in T^*,$ there is $\nu\geq \eta$ such that for each $s\in \Suc_{T^*}(\nu), (s, T^*_s)\| \lusim{a}(n)$ ($(s, T^*_s)$ decides $\lusim{a}(n)$).

For duration of the proof, define recursively the following for a tree  $T$  with stem $t$:
\begin{enumerate}
\item $\Suc^*_T(t)=\Suc_T(t),$
\item $\Lev^*_0(T)=\Suc_T(t),$
\item $\forall s\in \Lev^*_n(T), \Suc^*_T(s)=\Suc_T(s'),$ where $s' \geq s$ is minimal such that $|\Suc_T(s')|>1,$
\item $\Lev^*_{n+1}(T)=\{ s: s'\in \Lev^*_n(T), s\in \Suc^*_T(s')  \}.$
\end{enumerate}
The construction of $(t, T^*)$ is done by induction.
\\
{\bf Case $m=0$:} Set $T^0=T.$
\\
{\bf Case $m=n+1$:} Suppose the tree $T^n$ is constructed. For each $s\in \Lev^*_n(T^n),$ choose a condition $(f(s), S^s) \leq (s, T^n_s)$ such that
$(f(s), S^s) \| \lusim{a}(n).$ Let $T^{n+1}$ be the initial closure of the set $\{f(s):  s\in \Lev^*_n(T^n) \}$ together with $T^{n+1}_{f(s)}=S^s.$ The following is immediate:
\begin{enumerate}
\item $(t, T^{n+1}) \leq (t, T^n).$
\item $\Lev^*_n(T^{n+1})=\Lev^*_n(T^n).$
\item $(s, T^{n+1}_s) \| \lusim{a}(n),$ for each $s\in \Lev^*_n(T^{n+1}).$
\end{enumerate}
Set $T^*=\bigcap_{n<\omega}T^n.$ It is immediate that for each $n<\omega$
\begin{enumerate}
\item $(t, T^*) \leq (t, T^n).$
\item $\Lev^*_n(T^*)=\Lev^*_n(T^n).$
\item $(s, T^*_s) \| \lusim{a}(n),$ for each $s\in \Lev^*_n(T^*).$
\end{enumerate}
For each real $x$, we define the game $\mathcal{G}_x$ as follows:

$\hspace{2.5cm}$ $I:$$\hspace{.5cm}$ $a_0$ $\hspace{.5cm}$$\ldots$ $\hspace{.5cm}$ $\ldots$ $\hspace{.5cm}$ $a_n$ $\hspace{.5cm}$ $\ldots$

$\hspace{2.5cm}$ $II:$$\hspace{1.cm}$$s_0$ $\hspace{.5cm}$$\ldots$ $\hspace{.5cm}$ $\ldots$ $\hspace{.5cm}$ $s_n$ $\hspace{.5cm}$ $\ldots$

where $a_n$ and $s_n$, for  $n<\omega)$ are defined as follows:

player $I$ chooses $a_0\in [\Lev^*_0(T^*)]^{\leq \omega_1},$ and then player $II$ chooses $s_0\in \Lev^*_0(T^*)\setminus a_0.$ At step $n+1,$ player $I$ chooses $a_{n+1}\in [\Suc^*_{T^*}(s_n)]^{\leq \omega_1},$ and player $II$ replies by choosing some $s_{n+1}\in \Suc^*_{T^*}(s_n)\setminus a_{n+1}.$

Player $II$ wins iff for each $n<\omega, (s_n, T^*_{s_n})\Vdash$``$\lusim{a}(n)=x(n)$''. Note that if player $I$ wins the game, then he wins in a stage $n<\omega,$ and hence the game is open for one of the players, thus by Gale-Stewart theorem, there is a winning strategy for one of the players.
\begin{claim}
There is a real $x$ for which player $I$ does not have a winning strategy.
\end{claim}
\begin{proof}
Towards a contradiction, assume that player $I$ has a winning strategy $\sigma_x,$ for each real $x$. Build by induction the sequence $(s_n: n<\omega)$ with $s_n\in \Lev^*_n(T^*)$ as follows:
\\
{\bf Case $m=0$:} for each real $x$, set $a^x_0=\sigma_x(\langle \rangle).$ Let $a_0=\bigcup\{a^x_0: x$ is a real$ \}.$ Since $CH$ holds, $|a_0|\leq \aleph_1.$ Thus we can choose $s_0\in \Lev^*_0(T^*)\setminus a_0.$
\\
{\bf Case $m=n+1$:} for each real $x$, set $a^x_{n+1}=\sigma_x(s_0, ..., s_n).$ Let $a_{n+1}=\bigcup\{a^x_{n+1}: x$ is a real$ \}.$ Since $CH$ holds, $|a_{n+1}|\leq \aleph_1.$ Thus we can choose $s_{n+1}\in \Lev^*_{n+1}(T^*)\setminus a_{n+1}.$

Define the strategy $\tau$ for player $II$ to be the move $s_n$ in stage $n$ of the game. Since $s_{n}\in \Lev^*_{n}(T^*), (s_n, T^*_{s_n}) \| \lusim{a}.$ Thus we can define a real $x$ such that for each $n<\omega, (s_n, T^*_{s_n})\Vdash$``$\lusim{a}(n)=x(n)$''. But then player $II$ wins the game $\mathcal{G}_x$ using strategy $\tau,$ and we get a contradiction.
\end{proof}
\begin{claim}
Let $x$ be a real for which player $II$ has a winning strategy for the game $\mathcal{G}_x.$ Then there is a condition stronger than $(t, T^*)$ forcing $\lusim{a}=x.$
\end{claim}
\begin{proof}
Let $\tau$ be the winning strategy for player $II$ for the game $\mathcal{G}_x.$ For $n<\omega$ set
\begin{center}
$S_n=\{\tau(a_0, ..., a_n): \forall i\leq n, a_i\in [\Lev_i(T^*)]^{\leq \omega_1}       \}.$
\end{center}
Note that necessarily $|S_n|=\aleph_2,$ since otherwise player $I$ could have removed $S_n$ from the tree and win. Let $S^* \subseteq T^*$ be a tree satisfying $\forall n< \omega, \Lev^*_n(S^*)=S_n.$ Then $(t, S^*)\Vdash$``$\lusim{a}=x$''.
\end{proof}
The lemma follows.
\end{proof}

\section{An application of Namba forcing}
In this section, we give an application of Namba forcing. Recall that
\begin{theorem} (Jensen's covering lemma)
Assume $0^\sharp$ does not exist. Then for any uncountable set $X$ of ordinals, there exists a set of ordinals $Y\in L,$ the G\"{o}del's constructible universe, such that $X \subseteq Y$ and $|X|=|Y|$.
\end{theorem}
Now let $V$ be the generic extension of $L$ by Namba forcing, and let $C\in V$ be the added $\omega$-sequence cofinal in $\omega_2^L$. It is clear that $C$ can not be covered by a countable set from $L$. So in Jensen's covering lemma, we can not remove the uncountability assumption from the hypotheses.

\chapter{Prikry forcing}
Starting from a measurable cardinal $\k,$ we present a forcing construction, due to Prikry, which changes the cofinality of $\k$ into $\omega$ without collapsing cardinals.
\section{Measurable cardinals}
Let's start with the definition of a measurable cardinal.
\begin{definition}
$\k>\aleph_0$ is a measurable cardinal, if there exists a non-trivial elementary embedding $j: V \rightarrow M,$ from $V$ into an inner model $M$, such that $crit(j),$ the least ordinal moved by $j,$ is $\k$ and $^{\k}M \subseteq M.$
\end{definition}
Given a non-trivial elementary embedding $j$ as above, we can form $U=\{A \subseteq \k: \k\in j(A)\}.$ Then $U$ is a normal measure, i.e., it is a non-principal ultrafilter on $\k,$ and
\begin{enumerate}
\item $U$ is $\k$-complete: if $\l<\k$ and $\{A_\a: \a< \l\} \subseteq U,$ then $\bigcap_{\a<\l}A_\a \in U.$
\item $U$ is normal: if $\{A_\a: \a< \k\} \subseteq U,$ then $\bigtriangleup_{\a<\k}A_\a,$ the diagonal intersection of $A_\a$'s, is in $U$, where
    $\bigtriangleup_{\a<\k}A_\a=\{\xi<\k: \a<\xi \Rightarrow \xi\in A_\a\}.$
\end{enumerate}
We can also reverse the above construction, so that starting from any normal measure $U$ on an uncountable cardinal $\k,$ we can construct an inner model $M_U$ and an elementary embedding $j_U: V \rightarrow M_U$ such that $crit(j)=\k, ^{\k}M_U \subseteq M_U$ and $U=\{A \subseteq \k: \k\in j_U(A)    \}.$
\begin{lemma}
Let $\k$ be a measurable cardinal, $U$ be a normal measure on $\k, A\in U$ and let $f:[A]^{<\omega}\rightarrow \{0,1,2\}.$ There there is $B\in U, B \subseteq A$ which is homogeneous for $f$, i.e., for all $n<\omega, f \upharpoonright [B]^n$ is constant.
\end{lemma}
\section{Prikry forcing}
Throughout this section, fix a normal measure $U$ on a measurable cardinal $\k,$ which is derived from some elementary embedding $j: V \rightarrow M.$ The Prikry forcing $\PP_U$ consists of pairs $(s, A)$ where
\begin{enumerate}
\item $s\in [\k]^{<\omega},$
\item $A\in U,$
\item $\max(s) < \min(A).$
\end{enumerate}
The order relation is defined by $(s, A) \leq (t, B)$ iff
\begin{enumerate}
\item $s$ end extends $t,$
\item $ A \subseteq B$,
\item $s\setminus t \subseteq B.$
\end{enumerate}
The intuition behind this is that we are going to add an $\omega$-sequence $C$ cofinal in $\k;$ a condition $(s, A)$ carries the information that $s$ is an initial segment of this sequence, and the subsequent $C\setminus s$ must be chosen from $A$.
Let $G$ be $\PP_U$-generic over $V$. Set $C_G=\bigcup\{s: \exists A, (s,A)\in G\}.$
\begin{lemma}
$C_G$ is an $\omega$-sequence cofinal in $\k.$
\end{lemma}
\begin{proof}
It is clear that $C_G$ is a sequence of length at most $\omega.$ Given any $n<\omega,$ and any $\a<\k,$ the set
\begin{center}
$D_{n, \a}=\{(s, A): \len(s)>n$ and $\max(s)>\a   \}$
\end{center}
is dense in $\PP_U,$ from which it follows that $C_G$ is an $\omega$-sequence cofinal in $\k.$
\end{proof}
It is also clear that
\begin{center}
$G=\{(s, A): s$ is an initial segment of $C_G$ and $C_G\setminus s \subseteq A    \},$
\end{center}
and hence $V[G]=V[C_G].$ So we can talk about $\omega$-sequences from $\k$ being generic for the Prikry forcing $\PP_U;$ such sequences are called Prikry sequences. We are now going to show that forcing with $\PP_U$ preserves all cardinals.
\begin{lemma}
$\PP_U$ satisfies the $\k^+$-c.c.
\end{lemma}
\begin{proof}
First note that any two conditions $(s, A), (s, B)\in \PP_U$ are compatible, as witnesses by the common extension $(s, A\cap B).$ $[\k]^{<\omega}$ has cardinality $\k$, so any antichain contains at most $\k$ mutually incompatible members.
\end{proof}
It remains to show that cardinals $\leq \k$ are preserved. Define an auxiliary relation $\leq^*$ on $\PP_U,$ called the direct extension or the Prikry extension, by $(s, A)\leq^* (t, B)$ iff
\begin{enumerate}
\item $s=t,$
\item $ A \subseteq B$.
\end{enumerate}
 It is clear that $(\PP, \leq^*)$ is $\k$-closed, i.e., if $\l<\k$ and $(p_\a: \a<\l)$ is a $\leq^*$-decreasing sequence of conditions in $\PP_U,$ then there exists $p\in \PP_U$ which is a direct extension of each $p_\a, \a<\l.$ The main technical tool we will prove is the following
\begin{theorem}
$(\PP_U, \leq, \leq^*)$ satisfies the Prikry property: given any statement $\phi$ of the forcing language $(\PP_U, \leq),$ and any condition $(s, A)\in \PP_U,$ there exists $(s, B)\leq^* (s, A)$ such that $(s, B)$ decides $\phi$.
\end{theorem}
It is possible to use Lemma 6.1.2, to present a simple proof of Theorem 6.2.3; however, we will present a different proof, which  has the advantage that it can be applied for generalized Prikry like forcing notions. The main technical device is the diagonal intersection.

\begin{definition}
Suppose $(A_s: s\in [\k]^{<\omega})$ is such that each $A_s \subseteq \k.$ Then the diagonal intersection of this sequence is defined to be $\bigtriangleup_{s}A_s=\{\a< \k: \max(s) < \a \Rightarrow \a\in A_s\}.$
\end{definition}
\begin{lemma}
$(a)$ Suppose that each $A_s\in U.$ Then $A=\bigtriangleup_s A_s\in U,$ and for all $s, (s, A\setminus (\max(s)+1)) \leq (s, A_s).$

$(b)$ Let $D$ be a dense open subset of $\PP_U$. Then there exists $A\in U$ such that for all $s\in [\k]^{<\omega}, (\exists B(s, B)\in D \Leftrightarrow (s, A\setminus (\max(s)+1)\in D).$
\end{lemma}
\begin{proof}
$(a)$ To show that $A\in U,$ it suffices to show that $\k \in j(A),$ i.e.,
\begin{center}
$\forall s\in [j(\k)]^{<\omega} (\max(s)< \k \Rightarrow \k\in A_s)$
\end{center}
which is clear by our assumption. The second part is easily verified as $A\setminus (\max(s)+1) \subseteq A_s.$

$(b)$ For each $s,$ pick $A_s\in U$ such that $(s, A_s)\in D,$ if there is any, and $A_s=\k$ otherwise. Then $A=\bigtriangleup_s A_s$ is as required.
\end{proof}
We are now ready to complete the proof of Theorem 6.2.3.
\\
{\bf Proof of Theorem 6.2.3.} Assume towards a contradiction that there is no direct extension of $(s, A)$ which decides $\phi.$ The set $D=\{p\in \PP_U: p \| \phi\}$ is dense open, so by Lemma 6.2.5$(b),$ there exists $A^*\in U,$ such that for any $t \in [\k]^{<\omega} (\exists B, (t, B) \| \phi \Leftrightarrow (t, A^*\setminus (\max(t)+1))\| \phi).$ We may further suppose that $A^* \subseteq A\setminus (\max(s)+1)).$ For any $t\in [A^*]^{<\omega},$ we partition the set $A^*\setminus (\max(t)+1)$ into three sets

\hspace{1.5cm} $A^0_t=\{\a: (s^{\frown}t^{\frown}\a, A^*\setminus (\a+1)) \Vdash \phi         \},$

\hspace{1.5cm} $A^1_t=\{\a: (s^{\frown}t^{\frown}\a, A^*\setminus (\a+1)) \Vdash \neg\phi         \},$

\hspace{1.5cm} $A^2_t=\{\a: (s^{\frown}t^{\frown}\a, A^*\setminus (\a+1)) \nparallel \phi \}.$

For any $t$, there is a unique $i<3$ so that $A^i_t\in U,$ call it $A^*_t$. Also let $A^{**}=A^*\cap \bigtriangleup_t A^*_t.$ By our assumption, $(s, A^{**})$ does not decide $\phi.$ Let $(s^{\frown}t, B) \leq (s, A^{**})$ decides $\phi,$ where $\len(t)$ is minimal among such extensions. We will produce a shorter extension of $(s, A^{**})$ which also decides $\phi.$

Let us assume that $(s^{\frown}t, B) \Vdash$``$\phi$''. Note that $\len(t)>0,$ so we can write it as $t= u^{\frown}\a$.
 Then we have $\a\in A^*_u,$ and by our assumption, we must have $A^*_u=A^0_u.$ It follows from our choice of $A^0_u$ that
\begin{center}
$\forall \b\in A^{**}\setminus (\max(u)+1), (s^{\frown}u^{\frown}\b, A^{**}\setminus (\b+1))\Vdash$``$\phi$''.
\end{center}
Every extension of $(s^{\frown}u, A^{**}\setminus(\max(u)+1))$ is compatible with some condition of the form $(s^{\frown}u^{\frown}\b, A^{**}\setminus (\b+1)),$ where $\b\in A^{**}, \b > \max(u)+1,$ therefore $(s^{\frown}u, A^{**}\setminus(\max(u)+1))\Vdash$``$\phi$''. But $\len(u) < \len(t),$ and we get a contradiction with the minimal choice of $\len(t).$ \hfill$\Box$
\begin{lemma}
If $A\in V[G]$ is a bounded subset of $\k,$ then $A\in V.$
\end{lemma}
\begin{proof}
Let $p\in \PP_U,$ and $\l<\k$ be such that $p\Vdash$``$\lusim{A}$ is a subset of $\l$''. We build by induction a sequence $(p_\a: \a\leq \l)$ of direct extensions of $p$ such that:
\begin{enumerate}
\item $p_0=p,$
\item $\a< \b \Rightarrow p_\b \leq^* p_\a,$
\item $\forall \a<\l, p_{\a+1}\|$``$\a\in \lusim{A}$''.
\end{enumerate}
Then $A=\{\a<\l: p_\l\Vdash$``$\a\in \lusim{A}$''$  \},$ hence $A\in V.$
\end{proof}
It follows that cardinals $\leq \k$ are preserved in $V[G]$. Putting all of the above results together,  we have the following:
\begin{theorem}
Let $\k$ be a measurable cardinal, and $U$ be a normal measure on $\k.$ Then forcing with $\PP_U$ preserves cardinals and changes the cofinality of $\k$ to $\omega.$
\end{theorem}

\section{A geometric characterization of Prikry sequences}
We prove a characterization of Prikry generic $\omega$-sequences due to Mathias.
\begin{theorem}
Suppose that $U$ is a normal measure on  a measurable cardinal $\k,$ and let $\PP_U$ be the associated Prikry forcing. Then a sequence $C\in [\k]^{\omega},$ in any outer model of $V$, is $\PP_U$-generic over $V$ iff $\forall A\in V \exists m \forall n\geq m, C(n)\in A.$
\end{theorem}
\begin{proof}
First assume that $C$ is a Prikry generic sequence; so that $C=C_G$, for some $\PP_U$-generic $G$. Let $A\in U$ and $(s, B)\in \PP_U.$ The $(s, A\cap B)\in \PP_U$ extends $(s, B)$ and it forces ``$\lusim{C}\setminus s \subseteq A$''.

For the converse direction, let $G$ be the filter on $\PP_U$ generated by $C$, and let $D\in V$ be dense open in $\PP_U.$ By Lemma 6.2.5$(b),$ we can find $A\in U$ such that
\begin{center}
$\forall s\in [\k]^{<\omega} (\exists B (s, B)\in D \Leftrightarrow (s, A\setminus(\max(s)+1))\in D).$
\end{center}
For each $t\in [k]^{<\omega}$, define $f_t: [A\setminus (\max(t)+1)]^{<\omega} \rightarrow 2$ by
\begin{center}
 $f_t(s)\left\{
\begin{array}{l}
        0 \hspace{1.2cm} \text{ if } (t^{\frown}s, A\setminus (\max(s)+1))\in D,\\
        1 \hspace{1.2cm} \text{ if } \text{otherwise}.
     \end{array} \right.$
\end{center}
By Lemma 6.1.2, we can find $A_t\in U, A_t \subseteq A$ which is homogeneous for $f_t.$ Let $B=A\cap \bigtriangleup_t A_t.$
By our assumption, there is $m<\omega$ such that for all $n\geq m, C(n)\in B.$ Let $t=C \upharpoonright m,$ and note that if $n\geq m, C(n)\in A_t.$
As $D$ is dense, $(t, B)$ has some extension in $D$, and hence by our choice of $A,$ we can assume that it is of the form $(t^{\frown}s, B).$ But then
$s \subseteq A_t,$ so by homogeneity of $A_t,$ if $n=\len(t^{\frown}s)$ then $(C \upharpoonright n, B)\in D.$ But $(C \upharpoonright n, B)$ is also in $G$, hence $G$ meets $D$.
\end{proof}

\chapter{$HOD$ type models}
\section{G\"{o}del functions and inner models} The G\"{o}del functions are the following functional relations

\hspace{1.cm} $\mathcal{F}_1(x,y)=\{ x,y  \},$

\hspace{1.cm} $\mathcal{F}_2(x)=\in \upharpoonright x^2,$

\hspace{1.cm} $\mathcal{F}_3(x,y)=x\setminus y,$

\hspace{1.cm} $\mathcal{F}_4(x,y)=x \times y,$

\hspace{1.cm} $\mathcal{F}_5(x)=\bigcup x,$

\hspace{1.cm} $\mathcal{F}_6(x)=\dom(x),$

\hspace{1.cm} $\mathcal{F}_7(x)=\{(u,w,v): (u,v,w)\in x  \},$

\hspace{1.cm} $\mathcal{F}_8(x,y)=\{(v,u,w): (u,v,w)\in x  \}.$

\begin{lemma}
By composition one can get $x\cap y,$ $x^{-1}=\{(z,y): (y,z)\}$ and $\range(x).$
\end{lemma}
\begin{lemma}
Let $n\geq 2, i, j<n, i\neq j.$ Then
\begin{center}
$\{(a_0, ..., a_{n-1})\in y^n: (a_i, a_j)\in x \}= \mathcal{F}_{i,j,n}(x,y)$
\end{center}
is obtained from $x,y$ by composition of G\"{o}del functions.
\end{lemma}
The proof is by induction on $n.$
\begin{theorem}
Let $\phi(v_1, ..., v_n)$ a formula with free variables as shown; the value of $\phi$ in the structure $(x, \in \upharpoonright x^2)$ is given by a fixed functional $Val(\phi; x)$ which is a combination of G\"{o}del functions.
\end{theorem}
\begin{proof}
The proof is by induction on the length of formulas. For example:

{\bf Atomic case:} $Val(v_i\in v_j; x^n)=\mathcal{F}_{i,j,n}(x^n, x).$

{\bf $\exists$ case:} $Val(\exists v\phi)=\range(Val(\phi)).$
\end{proof}
\begin{theorem}
Let $M$ be a transitive class; assume that

$(a)$ $M$ is closed under G\"{o}del operations.

$(b)$ $\forall \xi, V_\xi\cap M$ is a set of $M$.

Then $M$ is an inner model.
\end{theorem}
\begin{proof}
The replacement axiom  should be the most difficult to prove. Now if
$\phi(v, w, v_1, ..., v_n)$ is a functional from parameters $a_1, ..., a_n$ and if $a$ is given, the image of $a$ under $\phi$ is a set in $V$, hence it is a subset of some $V_\xi$ and $M$, therefore it is a subset of some $x\in M.$ Therefore it is enough to apply comprehension; the same argument works for the axioms such as the power set. Finally everything backs down to comprehension. Thus we are given a set $a$, a formula $\psi(v, v_1, ..., v_n)$
and parameters $a_1, ..., a_n$. We apply the reflection principle to the class $M$ and the function relation $V_\xi\cap M;$ actually this is a generalized reflection principle:
\\
{\bf Generalized reflection principle:} Let $W$ be a class in $V$, let $W_\xi$ be a functional relation with domain $ON$ which is increasing and continuous. Let $\phi_1, ..., \phi_n$ be formulas of the language $ZFC$. For every ordinal $\xi$ there exists $\eta> \xi$ such that $W_\eta$ reflects $\phi_1, ..., \phi_n$ relative to $W$.

Thus if $\xi$ is chosen above the ranks of the parameters $a_1, ..., a_n$ and the given set $a$, we get for $b\in a$
\begin{center}
$M\models$``$\phi(b, a_1, ..., a_n)$'' iff $V_\xi\cap M\models$``$\phi(b, a_1, ..., a_n)$''.
\end{center}
Thus the required set is
\begin{center}
$Val(\phi; V_\xi\cap M) \cap (V_\xi \cap x) \times \{a_1\} \times ... \times \{a_n\}.$
\end{center}
But using closure under G\"{o}del operations
and the fact that $V_\xi\cap M$ is a set of $M$ we get $a$ a set in $M$.

The last thing to check is that $M$ contains all ordinals, otherwise $M\cap ON$ is an ordinal $\xi.$ But then
\begin{center}
$\xi=Val(\xi$ is an ordinal$; V_\xi\cap M)$
\end{center}
hence $\xi\in M,$ contradiction.
\end{proof}

\section{Ordinal definability} A set $a$ is definable if there exists a formula $\phi(v, v_1, ..., v_n),$ parameters $a_1, ..., a_n$ such that $a$ is the unique element  satisfying $\phi(a, a_1, ..., a_n)$ (this is equivalent to saying that for some formula $\psi(v, v_1, ..., v_n)$, $a$ is exactly $\{b: \psi(b, a_1, ..., a_n)$''$  \}$).

The definable sets (say without parameters) do not form a class. We shall see that if the parameters are chosen from another classes, it is the case.

We let $OD(X)$ be the class consisting of sets definable from ordinals and of a given class $X$.
\begin{lemma}
$OD(X)$ is closed under G\"{o}del functions.
\end{lemma}
This is easy to prove. For example if $x$ is defined by $\phi(v, a_1, ..., a_n)$ and $y$ is defined by $\psi(v, b_1, ..., b_n),$ then $x\times y$ is defined by $\Gamma(u,...):$
\begin{center}
$\exists v \exists v' \phi(v, a_1, ..., a_n) \wedge \psi(v', b_1, ..., b_n) \wedge u=(v, v').$
\end{center}
\begin{lemma}
Every element in $OD(X)$ can be obtained from elements of $ON\cup X$ and some $V_\xi$ by applying G\"{o}del functions.
\end{lemma}
\begin{proof}
This is because of the reflection principle. If $a=\{b: \phi(b, a_1, ..., a_n)\}.$ Then some $V_\xi$ reflects $\phi,$ large enough to include $a_1, ..., a_n.$ Thus $a=Val(\phi; V_\xi),$ and the result follows.
\end{proof}
For any class $X$, we let $X^{<\omega}$ be the class of finite sequences of elements of $X$.
\begin{theorem}
$OD(X^{<\omega})$ is closed under G\"{o}del functionals.
\end{theorem}
\begin{proof}
We have to perform some closure under G\"{o}del functions; but this has to be done inside set theory. Observe that if $b$ is definable from ordinals $\xi_1 > ... > \xi_p$ and members of $X^{<\omega},$ it is defined using a single ordinal $\omega^{\xi_1}+ ... + \omega^{\xi_p}$ and a single element of $X^{<\omega}.$ Hence Lemma 5.2.2 becomes `` $b$ is defined from some ordinal $\delta,$ some element of $X^{<\omega}$ and some $V_\xi$ by G\"{o}del functions''.
\end{proof}
A G\"{o}del term is a function defined on an integer $q+1$ such that $f(k)$ is
\begin{itemize}
\item an integer $0,1,2,$ or
\item A pair $(i, n), i=2,5,6,7,8, n<k,$ or
\item A tuple $(i, n, p), i=1,3,4, n,p<k.$
\end{itemize}
The value of a G\"{o}del term on $\delta, a, \xi$ is a function obtained by induction on $k\leq q$
\begin{center}
 $\nu(k)= \left\{
\begin{array}{l}
        \delta \hspace{3.5cm} \text{ if } f(k)=0,\\
        a \hspace{3.5cm} \text{ if } f(k)=1,\\
        V_\xi \hspace{3.4cm} \text{ if } f(k)=2,\\
        \mathcal{F}_i(\nu(n)) \hspace{2.45cm} \text{ if } i=2,5,6,7,8, f(k)=(i,n),\\
        \mathcal{F}_i(\nu(n), \nu(p)) \hspace{1.65cm} \text{ if } i=1,3,4, f(k)=(i,n, p).
     \end{array} \right.$
\end{center}
The final value of the term is $\nu(q).$ We enumerate all terms $( t_j: j\in\omega  ).$ Now the final value is definable from $\delta, a, \xi$ and the index of the term. So
\begin{center}
$(f\nu)[ON\times X^{<\omega}\times ON] \subseteq OD(X^{<\omega}).$
\end{center}
But by the lemma the converse also holds.
\begin{definition}
$HOD(X)$ is the set of elements $a$ such that $tcl(\{a\}) \subseteq OD(X).$
\end{definition}
It is a class as soon as $OD(X)$ is a class.
\begin{lemma}
$HOD(X)$  is closed under G\"{o}del functions.
\end{lemma}
\begin{theorem}
Let $X$ be a class. Assume that for any $\xi, V_\xi\cap X$ is a set of $OD(X)$. Then $HOD(X^{<\omega})$ is an inner model.
\end{theorem}
\begin{proof}
The only thing to prove is that $V_\xi \cap HOD(X^{<\omega})$ is an element of $HOD(X^{<\omega}).$ $HOD(X^{<\omega})$ is defined from $X$ by a formula $\phi(v)$ (with an extra predicate for $X$). Now we pick $\xi$ such that $V_\xi$ reflects $\phi$.
\begin{center}
$a=V_\xi \cap HOD(X^{<\omega})=\{ y\in V_\xi: V_\xi\models$``$\phi(y)$''$  \}.$
\end{center}
This gives the definition $V_\xi \cap HOD(X^{<\omega})$ from the structure $V_\xi$, this structure consists of $(V_\xi, \in \upharpoonright V_\xi, V_\xi\cap X),$ hence it is in $OD(X^{<\omega}).$

Finally $a$ is a subset of $HOD(X^{<\omega}),$ and hence of $OD(X^{<\omega}).$
\end{proof}
\begin{corollary}
We can consider the inner  models:
\begin{enumerate}
\item $HOD,$
\item $HOD(\{a\}),$
\item $HOD(N),$
\item $HOD(N^{\omega}),$
\item $HOD((N\cup tcl\{a\}^{<\omega}),$
\item $HOD(N\cup P(\omega)),$
\end{enumerate}
where $N$ is an inner model.
\end{corollary}
We note that as $N$ is an inner model, then
\begin{itemize}
\item $N^{<\omega} \subseteq N.$
\item Any element from $(N^\omega)^{<\omega}$ is definable from one element of $N^\omega.$
\item If $X=N\cup tcl\{a\}^{<\omega},$ any element of $X^{<\omega}$ is definable from one element of $N$ and one element of $tc\{a\}^{<\omega}.$
\item Any element in $(N\cup P(\omega))^{<\omega}$ is defined from an element of $N$ and one single real.
\end{itemize}

\section{The axiom of choice}
\begin{theorem}
$(a)$ $HOD$ satisfies $AC$,

$(b)$ If $M$ is an inner model which satisfies $AC$, then $HOD(M)$ also satisfies $AC$.
\end{theorem}
\begin{proof}
Recall that any element in $OD(X^{<\omega})$ is the final value of a G\"{o}del term on a triple $(\delta, a, \xi), \delta, \xi\in ON, a\in X^{<\omega}.$ So any element of $HOD(X^{<\omega})$ comes from a code $(t_i, \delta, a, \xi).$ If $X^{<\omega}$ is empty, this gives a way to well-order $HOD$. If $X$=inner model $M$, then given a set $u \in HOD(M)$, the set of codes $(t, \delta, a, \xi)$ form a set included in some $\omega\times \rho\times b\times \rho,$ some $b\in M, \rho\in ON,$ by replacement. This set is well-ordered and therefore $u$ is well-ordered.
\end{proof}
\begin{theorem} (Assume $AC$ holds)

$(a)$ Let $M$ be an inner model; then $HOD(M^\omega)$ satisfies dependent choice ($DC$).

$(b)$ Similarly $HOD(P(\omega))$ satisfies $DC$.
\end{theorem}
\begin{proof}
We only prove the second statement. We know that every element in $HOD(P(\omega))$ is the final value of a G\"{o}del term $t_i$ at some triple $(\delta, a, \xi), a\in P(\omega)^{<\omega}.$ Now any element of $P(\omega)^{<\omega}$ is coded by a single element of $P(\omega).$ Hence every element of $HOD(P(\omega))$ becomes a code which is a quadruple $(i, \xi, b, \delta),$ where $i\in \omega, \xi,\delta\in ON, b\in P(\omega).$ We now consider a binary relation $E$ on a set $X$, both lying in $HOD(P(\omega))$ such that
$\forall x\in X \exists y\in X, yEx$ holds. By applying choice we get a sequence $\langle (i_n, \xi_n, b_n, \delta_n): n<\omega   \rangle$ such that $\forall n, x_{n+1}Ex_n.$

Now $\langle b_n: n<\omega \rangle$ is coded by a single real $\beta$ of $P(\omega)$. We finally revise the definition of $i_n, \xi_n, \delta_n$ so as to obtain a new sequence  $\langle (i'_n, \xi'_n, b_n, \delta'_n): n<\omega   \rangle$, each time we take the first possible choice that allows an infinite sequence following what was built before. This gives the required sequence for $DC$.
\end{proof}

\section{Independence of $AC$} We force with conditions $p$ that are functions with finite domain from $\omega\times \omega$ to $2=\{0,1\}.$ This is a countable set of conditions, so that the cardinals are preserved in the generic extension $V[G]$.

The generic set $G$ defines a function $g:\omega\times \omega \rightarrow \{0,1\}$. For each $n, g_n$ is a subset of $\omega$ defined by
\begin{center}
$\{m: g(n, m)=1 \}.$
\end{center}
We let $a$ be $\{g_n: n<\omega     \}.$ It is easily seen that the $g_n$'s are distinct.
\begin{theorem}
In the model $M=(HOD(V \cup tcl\{a\}^{<\omega}))^{V[G]},$ the set $a$ is infinite and has no countable subset.
\end{theorem}
\begin{proof}
If the set $a$ was finite, it would be finite in $V[G]$ as well. Now if a function $f:\omega \rightarrow a$ is in the inner model $M$, it is definable from ordinals, a member of $V$, $a$ and an element of $a^{<\omega}$ say $u$.

We consider the first $k$ such that $f(k)$ is not in the range of $u$. This gives a $g_k$ definable from ordinals, on element $v$ of $V$, the set $a$ itself and an element of $a^{<\omega}.$ Actually the last element can be described by the function $\chi$ of $\omega^{<\omega}$ by labeling the $g_n$'s by their indices $n$. We pick a formula $\Phi$ that
\begin{center}
$\forall l \in \omega (l\in g_k \Leftrightarrow \Phi(l, \xi_1, ..., \xi_n, v, a, \chi)).$
\end{center}
If $\tau$ is a name for $a$ and $\sigma$ a name for $g_k$, then The following is forced by some condition $p_0$ in $G$:
\\
$(*)$ \hspace{2.5cm} $p_0\Vdash$``$\forall l\in \omega (l\in \lusim{g_k} \Leftrightarrow \Phi(l, \xi_1, ..., \xi_n, v, a, \chi)).$

We pick $k'$ such that $g_{k'}$ is not in the range of $u$, $k\neq k'$ and no integer $(k', i)$ appears in the domain of $p_0.$ An automorphism of
the set of forcing conditions is defined by exchanging $k$ and $k',$ formally
\begin{center}
 $\pi(p)(l,i)= \left\{
\begin{array}{l}
        p(k', i) \hspace{3.3cm} \text{ if } l=k,\\
        p(k,i) \hspace{3.4cm} \text{ if } l=k',\\
        p(l,i) \hspace{3.5cm} \text{ if }$ otherwise.$
     \end{array} \right.$
\end{center}
We note that $\pi(p_0)$ is compatible with $p_0$ and we can pick $p\leq p_0, \pi(p_0).$  Fix $G'$, generic so that $p\in G'$ and consider the models $V[G']$ and $V[\pi[G']].$  Because $(*)$ is forced, $g_k$ receives a definition in $V[G']$ through $\Phi, a, \chi.$ Now $V[G']$ and $V[\pi[G']]$ are the same generic model: only the order of the $g_n$'s differ. But in $V[\pi[G']]$ the $k$-th section is actually $g_{k'},$ so $g_{k'}= g_k$ and we get a contradiction.
\end{proof}

\end{document}